\documentclass[reqno]{conm-p-l}
\usepackage{amssymb,amsxtra}
\newtheorem{theorem}{Theorem}[section]
\newtheorem{proposition}[theorem]{Proposition}

\newtheorem{corollary}[theorem]{Corollary}

\newtheorem{definition}[theorem]{Definition}

\newtheorem{remark}[theorem]{Remark}

\numberwithin{equation}{subsection}

\begin{document}

\title[Truncated Infinitesimal Shifts and Quantized Universality of $\zeta(s)$]{Truncated Infinitesimal Shifts, Spectral Operators and Quantized Universality of the Riemann Zeta Function}

\author{Hafedh Herichi }
\address{Department of Mathematics, University of California, Riverside, CA 92521-0135}
\curraddr{}
\email{herichi@math.ucr.edu}
\thanks{}

\author{Michel\,L.\,Lapidus}
\address{Department of Mathematics, University of California, Riverside, CA 92521-0135}
\email{lapidus@math.ucr.edu}
\thanks{The work of M.\,L.\,Lapidus was partially supported by the US National Science Foundation under the research grant DMS-1107750, as well as by the Institut des Hautes Etudes Scientifiques (IHES) where the second author was a visiting professor in the Spring of 2012 while part of this work was completed.}

\date{August the 4th, 2012.}

\dedicatory{A Christophe Soul\'e, avec une profonde amiti\'e et admiration, \`a l'occasion de ses 60 ans}


\begin{abstract}
We survey some of the universality properties of the Riemann zeta function $\zeta(s)$ and then explain how to obtain a natural quantization of Voronin's universality theorem (and of its various extensions).\,Our work builds on the theory of complex fractal dimensions for fractal strings developed by the second author and M.\,van Frankenhuijsen in \cite{La-vF4}.\,It also makes an essential use of the functional analytic framework developed by the authors in [\textbf{HerLa1--3}] for rigorously studying the spectral operator $\mathfrak{a}$ (mapping the geometry onto the spectrum of generalized fractal strings), and the associated infinitesimal shift $\partial$ of the real line: $\mathfrak{a}=\zeta(\partial)$, in the sense of the functional calculus.\,In the quantization (or operator-valued) version of the universality theorem for the Riemann zeta function $\zeta(s)$ proposed here, the role played by the complex variable $s$ in the classical universality theorem is now played by the family of `truncated infinitesimal shifts' introduced in [\textbf{HerLa1--3}] to study the invertibility of the spectral operator in connection with a spectral reformulation of the Riemann hypothesis as an inverse spectral problem for fractal strings.\,This latter work provided an operator-theoretic version of the spectral reformulation obtained by the second author and H.\,Maier in \cite{LaMa2}.\,In the long term, our work (along with \cite{La5, La6}), is aimed in part at providing a natural quantization of various aspects of analytic number theory and arithmetic geometry.
\end{abstract}
\newpage

\begin{abstract}
Nous rappelons quelques unes des principales propri\'et\'es d'universalit\'e de la fonction z\^{e}ta de Riemann $\zeta(s)$.\,De plus, nous expliquons comment obtenir une quantification naturelle du th\'eor\`eme d'universalit\'e de Voronin (et de ses g\'en\'eralizations).\,Notre travail est bas\'e sur la th\'eorie des cordes fractales d\'evelopp\'ee par le deuxi\'eme auteur et M.\,van Frankenhuijsen dans \cite{La-vF4}.\\Nous utilisons \'egalement la th\'eorie d\'evelopp\'ee dans [\textbf{HerLa1--3}] par les auteurs de cet article pour \'etudier de fa\c{c}on rigoreuse l'op\'erateur spectral (qui relie la g\'eom\'etrie et le spectre des cordes fractales g\'en\'eraliz\'ees).\,Cet op\'erateur spectral est represent\'ee comme le compos\'e de la fonction z\^{e}ta de Riemann du `shift infinitesimal' $\partial:\,\mathfrak{a}=\zeta(\partial)$.\,Dans le processus du quantification du th\'eor\`eme d'universalit\'e de la fonction z\^{e}ta de Riemann, le r\^{o}le jou\'e par la variable $s$ (au sens classique du th\'eor\'eme d'universalit\'e) (dans le th\'eor\`eme classique d'universalit\'e) est jou\'e par la famille des shifts infinit\'esimaux tronqu\'es afin d'\'etudier l'op\'erateur spectral en lien avec la reformulation spectrale de l'hypoth\'ese de Riemann vue comme un probl\`eme spectral inverse pour les cordes fractales.\,Ce dernier r\'esultat fournit une version op\'eratorielle de la reformulation spectrale obtenue par le second auteur et H.\,Maier dans \cite{LaMa2}.\\Notre pr\'esent travail au long terme, ainsi que \cite{La5, La6}, a en partie pour but d'obtenir une quantification naturelle de divers aspects de la th\'eorie analytiques des nombres et de la g\'eom\'etrie arithm\'etiques. 
\end{abstract}
\maketitle
\newpage \setcounter{tocdepth}{2}
\tableofcontents
\begin{center}
\section{Introduction}
\end{center}

\hspace*{3mm}The universality of the Riemann zeta function states that any non-vanishing analytic function can be approximated uniformly by certain purely imaginary shifts of the zeta function in the right half of the critical strip \cite{Vor2}.\,It was discovered by S.\,M.\,Voronin in 1975.\,Several improvements of Voronin's theorem are given in \cite{Bag1, Lau1, Rei1, Rei2}.\,Further extensions of this theorem to other classes of zeta functions can be found in \cite{Emi, Lau2, LauMa1, LauMa2, LauMaSt, LauSlez, LauSt, St1}.\,In the first part of the present paper, as well as in several appendices, we survey some of these results and discuss their significance for the Riemann zeta function $\zeta$ and other $L$-functions.\\

\hspace*{3mm}In the second part of the paper, we focus on the Riemann zeta function $\zeta(s)$ and its operator (or `quantum') analog, and propose a quantum (or operator-valued) version of the universality theorem and some of its extensions.\,More specifically, in their development of the theory of complex dimensions in fractal geometry and number theory, the spectral operator was introduced heuristically by the second author and M.\,van Frankenhuijsen as a map that sends the geometry of fractal strings to their spectra \cite{La-vF3, La-vF4}.\,A detailed, rigorous functional analytic study of this map was provided by the authors in \cite{HerLa1}.\\

\hspace*{3mm}In this paper, we discuss the properties of the family of `\emph{truncated infinitesimal shifts}' of the real line  and of the associated `\emph{truncated spectral operators}'.\,(See also \cite{HerLa2, HerLa3, HerLa4}.)\,These truncated operators were introduced in \cite{HerLa1} to study the invertibility of the spectral operator and obtain an operator-theoretic version of the spectral reformulation of the Riemann hypothesis (obtained by the second author and H.\,Maier in \cite{LaMa1, LaMa2}) as an inverse spectral problem for fractal strings.\,In particular, we show that $\mathfrak{a}=\zeta(\partial)$, where $\partial$ is the `infinitesimal shift' (the generator of the translations on the real line), viewed as an unbounded normal operator acting on a suitable scale of Hilbert spaces indexed by the (Minkowski) fractal dimension of the underlying (generalized) fractal strings.\\

\hspace*{3mm}Moreover, using tools from the functional calculus and our detailed study of the spectra of the operators involved, we show that any non-vanishing holomorphic function of the truncated infinitesimal shifts can be approximated by imaginary translates of the truncated spectral operators.\,This latter result provides a `natural quantization' of Voronin's theorem (and its extensions) about the universality of the Riemann zeta function.\,We conclude that, in some sense, arbitrarily small scaled copies of the spectral operator are encoded within itself.\,Therefore, we deduce that the spectral operator can emulate any type of complex behavior and that it is `chaotic'.\,In the long term, the theory developed in the present paper and in \cite{HerLa1, HerLa2, HerLa3, HerLa4}, along with the work in \cite{La5, La6}, is aimed in part at providing a natural quantization of various aspects of analytic (and algebraic) number theory and arithmetic geometry.\\

\hspace*{3mm}The rest of this paper is organized as follows:\,in \S2, we discuss the classical universality property of the Riemann zeta function and some of its applications.\,In \S3, we briefly review the theory of (generalized) fractal strings and the associated complex dimensions and explicit formulas, as developed in \cite{La-vF4}.\,We also recall the heuristic definition of the spectral operator introduced in [\textbf{La-vF3}, \textbf{La-vF4}].\,In \S4, we develop the rigorous functional analytic framework of \cite{HerLa1}; we define and study, in particular, the infinitesimal shift $\partial$ and the spectral operator $\mathfrak{a}$, along with their truncated versions, $\partial^{(T)}$ and $\mathfrak{a}^{(T)}$ (for $T>0$).\,In \S5, we provide our quantization (or operator-valued version) of Voronin's theorem for the universality of $\zeta(s)$, along with its natural generalizations in this context.\,It is noteworthy that in this `quantization process', the complex variable $s$ is replaced not by $\partial$ (the infinitesimal shift), as one might reasonably expect, but by the family of truncated infinitesimal shifts $\{\partial^{(T)}\}_{T>0}$.\,In \S6, we propose several possible directions for future research in this area.\,Finally, in three appendices, we provide some additional information and references about the origins of universality, as well as about the extensions of Voronin's universality theorem to other arithmetic zeta functions (including the Dirichlet $L$-functions and the $L$-functions associated with certain modular forms).\\

\section{Universality of the Riemann Zeta Function}

\hspace*{3mm}In this section, we recall some of the basic properties of the Riemann zeta function (in \S2.1), and then focus (in \S2.2) on the universality of $\zeta(s)$ (among all non-vanishing holomorphic functions).\,We also briefly discuss (in \S2.3) some of the mathematical and physical applications of universality.\\

\subsection{The Riemann zeta function $\zeta(s)$}
\hspace*{3mm}The Riemann zeta function is defined as the complex-valued function
\begin{equation}
\zeta(s)=\sum_{n=1}^{\infty}n^{-s},\mbox{\quad for\,\,}Re(s)>1.\label{Zeta}
\end{equation}

\hspace*{3mm}In 1737, Euler showed that this Dirichlet series\footnote{In 1740, Euler initiated the study of the Dirichlet series given in Equation (\ref{Zeta}) for the special case when the complex number $s$ is a positive integer.\,Later on, his work was extended by Chebychev to $Re(s)>1$, where $s\in\mathbb{C}$.} can be expressed in terms of an infinite product over the set $\mathcal{P}$ of all the prime numbers:
\begin{equation}\label{Eq:Rmprd}
\zeta(s)=\sum_{n=1}^{\infty}n^{-s}=\prod_{p\in\mathcal{P}}\frac{1}{1-p^{-s}},\mbox{\quad for $Re(s)>1$}.
\end{equation}
Note that Equation (\ref{Eq:Rmprd}) shows that the Riemann zeta function carries information about the primes, which are encoded in its Euler product.\\

\hspace*{3mm}In 1858, Riemann showed in \cite{Rie} that this function has a meromorphic continuation to all of $\mathbb{C}$ with a single (and simple) pole at $s=1$, which satisfies the \emph{functional equation}
\begin{equation}
\xi(s)=\xi(1-s),\mbox{\,}\,s\in \mathbb{C},\label{Eq:fE}
\end{equation}
where
\begin{equation}
\xi(s):=\pi^{-\frac{s}{2}}\Gamma(\frac{s}{2})\zeta(s)\label{Eq:CZ}
\end{equation}
is the \emph{completed} $($or \emph{global}$)$ Riemann zeta function (Here, $\Gamma$ denotes the classic gamma function.)\,Note that the trivial zeros of $\zeta(s)$ at $s=-2n$ for $n=1, 2, 3, ...,$ correspond to the poles of the gamma function $\Gamma(\frac{s}{2})$.\,Riemann also conjectured that the nontrivial $($or \emph{critical}$)$ zeros of $\zeta(s)$ $($i.e., the zeros of $\zeta(s)$ which are located in the critical strip $0<Re(s)<1$$)$ all lie on the \emph{critical line} $Re(s)=\frac{1}{2}$.\,This famous conjecture is known as the \emph{Riemann hypothesis}.\\

\hspace*{3mm}It is well known that the Euler product in Equation (\ref{Eq:Rmprd}) converges absolutely to $\zeta(s)$ for $Re(s)>1$ and also uniformly on any compact subset of the half-plane $Re(s)>1$.\,We note that the Euler product (or an appropriate substitute thereof) can be useful even in the critical strip $\{0<Re(s)<1\}$ where it does not converge.\,For example, it turns out that a suitable truncated version of this Euler product, namely,
\begin{equation}
\zeta_{N}(s)=\prod_{p\leq N}\big(1-p^{-s}\big)^{-1} \label{TruncEulProd}
\end{equation}
(possibly suitably randomized), played a key role in Voronin's proof of the universality of the Riemann zeta function.\,This idea was due to Harald Bohr's earlier work on the density of the sets consisting of the ranges of $\zeta(s)$ on the vertical lines $L_{s}=\{s\in \mathbb{C}:\,Re(s)=c\}$, with $\frac{1}{2}<c<1$.\,Although it is a known fact that the Riemann zeta function's Euler product (see the right-hand side of Equation (\ref{Eq:Rmprd})) does not converge to $\zeta(s)$ inside the critical strip (in particular, inside the right-hand side of the critical strip $\frac{1}{2}<Re(s)<1$), Voronin showed that a `suitable' truncated version of this Euler product can be used to approximate $\zeta(s)$ on the right-hand side of the critical strip, i.e., on the half-plane $\{\frac{1}{2}<Re(s)<1\}$ (see \cite{Vor2}).\\

\hspace*{3mm}The Riemann zeta function has a number of applications in the mathematical and physical sciences (also, in biology and economics).\,For instance, in analytic number theory, the identity given in Equation (\ref{Eq:Rmprd}) can be used to show that there are infinitely many primes among the integers.\,The Riemann zeta function plays a key role in describing the distribution of the prime numbers.\,It also appears in applied statistics (in the Zipf--Mandelbrot law), as well as in quantum field theory (in the calculation of the Casimir effect).\,In fractal geometry (for instance in the theory of complex dimensions), the Riemann zeta function naturally occurs as a multiplicative factor in the formula relating the geometry and spectra of fractal strings via their geometric and spectral zeta functions (see \cite{La2, La3, LaMa2, La-vF4}).\,The Riemann zeta function has several other interesting applications in physics.\,(See, e.g., \cite{Tit, Edw, Ing, Ivi, Pat, KarVo, Ser, La-vF4, La5} for a more detailed discussion of the theory of the Riemann zeta function and some of its applications; see also Riemann's 1858 original paper \cite{Rie}.)\,It will be shown in the next section that this function is also the first explicit \emph{universal}\footnote{See \S 2.2 for an explanation of the notion of `universality' of $\zeta(s)$ and see also Appendix A about the origins of universality in the mathematical literature.} object that was discovered in the mathematical sciences.

\subsection{Voronin's original universality theorem for $\zeta(s)$ }

\hspace*{3mm}The universality of the Riemann zeta function was established by S.\,M.\,Voronin in 1975.\,This  important property of the Riemann zeta function states that any non-vanishing (i.e., nowhere vanishing) analytic function can be uniformly  approximated by certain purely imaginary shifts of the zeta function in the right half of the critical strip \cite{Vor2}.\,The following is Voronin's original universality theorem:\footnote{We refer to Appendix B for several extensions of Voronin's universality theorem to other elements of the Selberg class of zeta functions.}

\begin{theorem}\label{VorThm}
Let $0<r<\frac{1}{4}$.\,Suppose that $g(s)$ is a non-vanishing continous function on the disk $D=\{s\in\mathbb{C}:\,|s|\leq r\}$, which is analytic $($i.e., holomorphic$)$ in the interior \r{D} $=\{s\in\mathbb{C}:\,|s|<r\}$ of this disk.\,Then, for any $\epsilon>0$, there exists $\tau\geq0$ such that
\begin{equation}
J(\tau):=\sup_{|s|\leq r}|g(s)-\zeta(s+\frac{3}{4}+i\tau)|<\epsilon.
\end{equation}

\hspace*{3mm}Moreover, the set of such $\tau$'s is infinite.\,In fact, it has a positive lower density; i.e., 
\begin{equation}\label{Eq1}
\liminf_{T\to \infty}\frac{1}{T}vol_{1}(\{\tau\in[0,T]:\,J(\tau)<\epsilon\})>0,
\end{equation}
where $vol_{1}$ denotes the Lebesgue measure on $\mathbb{R}$.
\end{theorem}

\hspace*{3mm}The Riemann zeta function is said to be \emph{universal} since suitable approximate translates (or imaginary shifts) of this function uniformly approximate any analytic target function satisfying the hypothesis given above in Theorem \ref{VorThm}.\,The domain of the uniform approximation of the admissible target function is called the \emph{strip of universality}.\\

\hspace*{3mm}The strongest version of Voronin's theorem (the extended Voronin theorem) is due to Reich and Bagchi (see \cite{KarVo, Lau1, Lau2, St1}); it is given in the following result:

\begin{theorem}\label{Thm:ExtVor}
\hspace*{3mm}Let $K$ be any compact subset of the right critical strip $\{\frac{1}{2}<Re(s)<1\}$, with connected complement in $\mathbb{C}$.\,Let $g:K\to \mathbb{C}$ be a non-vanishing continuous function that is holomorphic in the interior of $K$ $($which may be empty$)$.\\Then, given any $\epsilon>0$, there exists $\tau\geq 0$ $($depending only on $\epsilon$$)$ such that

\begin{equation}
J_{sc}(\tau):=\sup_{s\in K}|g(s)-\zeta(s+i\tau)|\leq \epsilon.
\end{equation}

\hspace*{3mm}Moreover, the set of such $\tau$'s is infinite.\,In fact, it has a positive lower density $($in the precise sense of Equation $($\ref{Eq1}$)$ above, but with $J_{sc}(\tau)$ instead of $J(\tau)$$)$.
\end{theorem}

\begin{remark}\label{Rk:2.4}
The topological condition on the subset $K$ cannot be significantly refined.\,Furthermore, the condition according to which the target function $g(s)$ has to be non-vanishing in the disk $D$ is crucial and cannot be omitted.\,Indeed, an application of Rouch\'e's theorem shows that if such a condition were violated, then one would obtain a contradiction to the density theorem for the zeros of the Riemann zeta function $($a detailed proof of this statement can be found in \cite{St1}$)$.\,This latter result shows that the universality of $\zeta(s)$ is intimately connected to the location of the critical zeros of the Riemann zeta function which are conjectured $($under the Riemann hypothesis$)$ to be located on the vertical line $Re(s)=\frac{1}{2}$.\,We will see in the next subsection that this connection was strengthened in Bagchi's theorem, which relates the universality of $\zeta(s)$ to the Riemann hypothesis.\,$($See Theorem \ref{RHUniv} below.$)$
\end{remark}

\begin{remark}
The proof of Voronin's original theorem $($Theorem \ref{VorThm}$)$ about the universality of the Riemann zeta function, along with that of its various extensions $($including Theorem \ref{Thm:ExtVor}$)$, is \emph{ineffective}, in the sense that it does not give any information about how fast a target function $g(s)$ can be approximated by imaginary translates of $\zeta(s)$ within a given range.\,Moreover, the known proofs of universality neither provide a good estimate for the `first' approximating shift $\tau$ nor a specific bound for the positive lower density of the admissible shifts.\,In this paper, we will not discuss these issues any further.\,Several attempts at solving the effectivity problem for the universality theorems were made by Montgomery, Good and Garunk\v{s}tis $($see \cite{St1, Gar, Goo, Mont}$)$.
\end{remark}

\hspace*{3mm}We refer the interested reader to the books \cite{KarVo, Lau1, St1} for additional information concerning the universality properties of the Riemann zeta function and of other arithmetic zeta functions.\,(See also Appendix B below, provided in \S8.)

\subsection{Some applications of the universality of $\zeta(s)$}

\hspace*{3mm}The universality theorem (Theorem \ref{Thm:ExtVor}) for the Riemann zeta function has several interesting applications, which are related to functional independence, the critical zeros of the Riemann zeta function and therefore, to the Riemann hypothesis (this latter result was obtained in the work of Bagchi \cite{Bag1, Bag2}), the approximation of certain target functions by Taylor polynomials of zeta functions \cite{GauCl}, and also to path integrals in quantum theory (see \cite{BitKuRen}).

\newpage
\begin{center}
\textbf{Extension of the Bohr--Courant density theorem}
\end{center}
The universality theorem (Theorem \ref{Thm:ExtVor}) for the Riemann zeta function provides an extension (due to Voronin in \cite{Vor1}) of the Bohr--Courant classical result \cite{BohCou} about the density of the range of $\zeta(s)$ on an arbitrary vertical line contained in the right critical strip $\{s\in\mathbb{C}:\,\frac{1}{2}<Re(s)<1\}$.\,The following result (also due to Voronin in \cite{Vor1}) can either be deduced from the universality theorem or proved directly (as was done originally in \cite{Vor1}):

\begin{theorem}\label{Thmdens}
Let $x\in(\frac{1}{2}, 1)$ be fixed.\,$($Here, $Re(s)=x$ and $Im(s)=y$.$)$\,Then, for any integer $n\geq1$, the sets
\begin{equation}
\big\{\big((\log\zeta(x+iy)), (\log\zeta(x+iy))',...,(\log\zeta(x+iy))^{(n-1)}\big):\,y\in\mathbb{R}\big\}
\end{equation}
and
\begin{equation}
\big\{\zeta(x+iy), (\zeta(x+iy))',...,(\zeta(x+iy))^{(n-1)}\big):\,y\in\mathbb{R}\big\}
\end{equation}
are dense in $\mathbb{C}^{n}$.\footnote{Here and in the sequel, we use the standard notation for the complex derivative (and higher order derivatives) of an analytic function.}
\end{theorem}

\begin{center}
\textbf{Functional independence and hypertranscendence of $\zeta(s)$}
\end{center}
Furthermore, the universality of $\zeta(s)$ implies functional independence.\,In 1887, H\"older proved the functional independence of the gamma function, i.e., that the gamma function $\Gamma(s)$ does not satisfy any algebraic differential equation.\,In other words, for any integer $n\geq 1$, there exists no non-trivial polynomial $P$  such that 
\begin{equation}
P(\Gamma(s), \Gamma'(s),...,\Gamma^{(n-1)}(s))=0.
\end{equation}

\hspace*{3mm}In 1900, and motivated by this fact, Hilbert suggested at the International Congress for Mathematicians (ICM) in Paris that the algebraic differential independence of the Riemann zeta function \footnote{We would like to attract the reader's attention to the fact that the original proof of the hypertranscendence of the Riemann zeta function was due to Stadigh.\,Later on, such a proof was given in a more general setting by Ostrowski, using a different mathematical approach (see \cite{Os}).} can be proved using H\"older's above result and also the functional equation for $\zeta(s)$.\,In 1973, and using a `suitable' version of the universality theorem for the Riemann zeta function, S.\,M.\,Voronin \cite{Vor3, Vor4} has established the functional independence of $\zeta(s)$, as we now explain:

\begin{theorem}\label{FuncInd}
Let $z=(z_{0}, z_{1},..., z_{n-1})\in \mathbb{C}^{n}$ and let $N$ be the a nonnegative integer.\,If $F_{0}(z)$, $F_{1}(z)$,..., $F_{N}(z)$ are continuous functions, not all vanishing simultaneously, then there exists some $s\in\mathbb{C}$ such that
\begin{equation}
\sum_{k=0}^{N}s^{k}F_{k}(\zeta(s), \zeta'(s),...,\zeta^{(n-1)}(s))\ne0.
\end{equation} 
\end{theorem}

\hspace*{3mm}Note that Theorem \ref{FuncInd} implies that the Riemann zeta function does not satisfy any algebraic differential equation.\,As a result, $\zeta(s)$ is \emph{hypertranscendental}.\footnote{Theorem \ref{FuncInd} provides an alternative proof of the solution of one of Hilbert's famous problems proposed in 1900 at the International Congress of Mathematicians in Paris.}\\

\begin{center}
\textbf{Universality of $\zeta(s)$ and the Riemann hypothesis}
\end{center}
\hspace*{3mm}In \cite{Bag1, Bag2}, Bagchi showed that there is a connection between the location of the critical zeros of the Riemann zeta function and the universality of this function.\,The corresponding result can be stated as follows:

\begin{theorem}\label{RHUniv}
Let $K$ be a compact subset of the vertical strip $\{\frac{1}{2}<x<1\}\times\mathbb{R}$, with connected complement.\,Then, for any $\epsilon>0$, we have
\begin{equation}\label{Eq2}
\liminf_{T\to\infty}\frac{1}{T}vol_{1}\big(\big\{\tau\in[0,T]:\,\sup_{s\in K}|\zeta(s)-\zeta(s+i\tau)|<\epsilon\big\}\big)>0
\end{equation}
if and only if the Riemann hypothesis $($RH$)$ is true.\,Here, as before, $vol_{1}$ denotes the Lebesgue measure on $\mathbb{R}$.
\end{theorem}
As a result, the universality of $\zeta(s)$ can be used to reformulate RH.\\

\begin{remark}
Note that if the Riemann hypothesis is true, then $\zeta(s)$ is a non-vanishing analytic function in the right critical strip $\{\frac{1}{2}<Re(s)<1$\}.\,Hence, the fact that Equation $($\ref{Eq2}$)$ holds follows at once from Voronin's extended theorem $($Theorem \ref{Thm:ExtVor}$)$ applied to $g(s):=\zeta(s)$.\,The converse direction can be established by reasoning by contradiction and applying Rouch\'e's theorem $($from complex analysis$)$ about the number of zeros of the perturbation of an analytic function. 
\end{remark}

\begin{center}
\textbf{Approximation by Taylor polynomials of $\zeta(s)$}
\end{center}

\hspace*{3mm}Another interesting application of the universality theorem for the Riemann zeta function was given in the work of M. Gauthier and R. Clouatre \cite{GauCl}.\,Using Voronin's universality theorem, these authors showed that every holomorphic function on a compact subset $K$ of the complex plane having a connected complement $K^{c}$ can be uniformly approximated by vertical translates of Taylor polynomials of $\zeta(s)$.\footnote{Their result was  suggested in 2006 by Walter Hayman as a possible step toward attempting to prove the Riemann hypothesis (see \cite{GauCl}).}\,Denote by $H(K)$ the set of all complex-valued holomorphic functions in an open neighborhood $\omega$ of $K\subset\mathbb{C}$ and by $T^{f}_{n,\,a}(z):=\sum_{k=0}^{n}\frac{f^{(k)}}{k!}(z-a)^{k}$, the $n$-th Taylor polynomial of $f$ centered at $a$, where $z\in \mathbb{C}$.\,Then we have the following result:

\begin{theorem}
Let $K\subset\mathbb{C}$ with $K^{c}$ connected.\,Let $g\in H(K)$ and $\epsilon>0$.\,Then, for each $z_{0}\in \omega\cap K^{c}$, there exists $\tau\in \mathbb{R}$ and $n\in \mathbb{N}$ such that
\begin{equation}
\sup_{z\in K}|g(z)-T^{\zeta}_{n,\,z_{0}+i\tau}(z+i\tau)|<\epsilon.
\end{equation}
\end{theorem}

\begin{center}
\textbf{Path integrals, quantum theory and Voronin's universality theorem}
\end{center}

\hspace*{3mm}A physical application of Voronin's theorem about the universality of the Riemann zeta function was obtained by K.\,M.\,Bitar, N.\,N.\,Khuri and H.\,C.\,Ren.\,Within their framework, Voronin's universality theorem was used to explore a new numerical approach (via suitable discrete discretizations) to path integrals in quantum mechanics (see \cite{BitKuRen}).\\

\hspace*{3mm}We will see in \S5 that a `\emph{quantization}' of the universality of the Riemann zeta function was obtained in \cite{HerLa1}.\,It enables us to obtain an operator-theoretic extension of Theorem \ref{Thm:ExtVor} on the universality of the Riemann zeta function.\,This quantization is obtained in terms of a `suitable' truncated version $\mathfrak{a}_{c}^{(T)}$ of the spectral operator $\mathfrak{a}_{c}$ (see \S4 for a detailed discussion of the properties of $\mathfrak{a}_{c}^{(T)}$), a map that relates the geometry of fractal strings to their spectra.\,The study of the spectral operator was suggested by M.\,L.\,Lapidus and M.\,van Frankenhuijsen in their development of the theory of complex dimensions in fractal geometry \cite{La-vF1, La-vF2, La-vF3, La-vF4}.\,Later on, it was thoroughly investigated (within a rigorous functional analytic framework) in \cite{HerLa1} and surveyed in the papers \cite{HerLa2, HerLa3}.\,In the next section, we start by introducing the class of generalized fractal strings and then define the spectral operator for fractal strings.

\section{Generalized Fractal Strings and the Spectral Operator $\mathfrak{a}=\zeta(\partial_{c})$}

\hspace*{3mm}In this section, we first recall (in \S3.1) the notion of generalized fractal string introduced and used extensively in \cite{La-vF1, La-vF2, La-vF3, La-vF4}, for example, in order to obtain general explicit formulas applicable to various aspects of number theory, fractal geometry, dynamical systems and spectral geometry.\,In \S 3.2, after having recalled the original heuristic definition of the spectral operator $\mathfrak{a}$ (as given in \cite{La-vF3, La-vF4}), we rigorously define $\mathfrak{a}=\mathfrak{a}_{c}$ as well as the infinitesimal shift $\partial=\partial_{c}$ as unbounded normal operators acting on a scale of Hilbert spaces $\mathbb{H}_{c}$ parametrized by a nonnegative real number $c$, as was done in \cite{HerLa1, HerLa2, HerLa3}.\,Finally in \S3.3, we briefly discuss some of the properties of the infinitesimal shifts and of the associated translation semigroups.

\subsection{Generalized fractal strings and explicit formulas}
\hspace*{3mm}A \emph{generalized fractal string} $\eta$ is defined as a local positive or complex measure on $(0,+\infty)$ satisfying $|\eta|(0,x_{0})=0$,\footnote{Here, the positive (local) measure $|\eta|$ is the total variation measure of $\eta$.\,(For a review of standard measure theory, see, e.g., \cite{Coh, Fo}.)} for some $x_{0}>0$.\,For instance, if we consider the ordinary fractal string\footnote{An ordinary fractal string $\mathcal{L}$ is a bounded open subset of the real line.\,Such a set consists of countably many (bounded) open intervals $\{(a_{j}, b_{j})\}_{j}^{\infty}$, the lengths of which are denoted by $l_{j}=b_{j}-a_{j}$, for each $j\geq 1$.\,We simply write $\mathcal{L}=\{l_{j}\}_{j=1}^{\infty}$.\,For instance, the Cantor string, defined as the complement of the ternary Cantor set in the interval $[0,1]$, is an example of an ordinary fractal string.\,It is explicitly defined as $\mathcal{L}_{CS}=\{l_{j}\}_{j=1}^{\infty}$, where the corresponding lengths are $l_{j}=3^{-j}$ with the corresponding multiplicities $w_{j}=2^{j-1}$, for each $j\geq 1$.} $\mathcal{L}=\{l_{j}\}_{j=1}^{\infty}$ with integral multiplicities $w_{j}$, then the measure $\eta_{\mathcal{L}}$ associated to $\mathcal{L}$ is a standard example of a generalized fractal string:
\begin{equation}
\eta_{\mathcal{L}}:=\sum_{j=1}^{\infty}w_{j}\delta_{\{l_{j}^{-1}\}},
\end{equation}
where $\delta_{\{x\}}$ is the Dirac delta measure (or the unit point mass) concentrated at $x>0$ and $l_{j}$ are the scales associated to the connected components (or open intervals) constituting the ordinary fractal string $\mathcal{L}$.\,We refer the reader to \cite{La-vF1, La-vF2, La-vF3, La-vF4} for more information about the theory of ordinary fractal strings and many of its applications.

\begin{remark}
In contrast to an ordinary fractal string, the multiplicities $w_{j}$ of a generalized fractal string $\eta$ may be non-integral.\,For example, the prime string 
\begin{equation}
\eta:=\sum_{m\geq1,\,p}(\log p) \delta_{\{p^{m}\}} 
\end{equation}
$($where p runs over all prime numbers$)$, is clearly a generalized $($and not an ordinary$)$ fractal string.\,Furthermore, observe that the multiplicities need not be positive numbers either.\,A simple example of this situation is the M\"obius string
\begin{equation}\label{GfMob}
\eta_{\mu}:=\sum_{j=1}^{\infty}\mu(j)\delta_{\{j\}},
\end{equation}
where $\mu(j)$ is the M\"obius function.\footnote{The M\"obius function $\mu(j)$ equals 1 if $j$ is a square-free positive integer with an even number of prime factors, -1 if $j$ is a square-free positive integer with an odd number of prime factors, and 0 if $j$ is not square-free.}\,This string can be viewed as the measure associated to the fractal string $\mathcal{L}_{\mu}=\{j^{-1}\}_{j=1}^{\infty}$ with real multiplicities $w_{j}=\mu(j)$.\,Note that this string is not an ordinary fractal string.\,As a result, the use of the word `generalized' is well justified for this class of strings.
\end{remark}

\hspace*{3mm}The \emph{dimension} $D_{\eta}$ of a generalized fractal string $\eta$ is the abscissa of convergence of the Dirichlet integral $\int_{0}^{\infty}x^{-s}\eta(dx)$:
\begin{equation}
D_{\eta}:=\inf\biggl\{\sigma\in\mathbb{R}:\,\int_{0}^{\infty}x^{-\sigma}|\eta|(dx)<\infty\biggl\},
\end{equation}
The geometric \emph{counting function} of $\eta$ is 
\begin{equation}
N_{\eta}(x):=\int_{0}^{x}\eta(dx)=\frac{1}{2}(\eta(0,x)+\eta(0,x]).
\end{equation}
The \emph{geometric} zeta function associated to $\eta$ is the Mellin transform of $\eta$:
\begin{equation}
\zeta_{\eta}(s):=\int_{0}^{\infty}x^{-s}\eta(dx) \mbox{\quad for\,}Re(s)>D_{\eta}.
\end{equation}

\hspace*{3mm}From now on, we will assume that $\zeta_{\eta}$ has a meromorphic extension to a suitable open connected neighborhood $\mathcal{W}$ of the half-plane of absolute convergence of $\zeta_{\eta}(s)$ (i.e., $\{Re(s)>D_{\eta}\}$).\,The set $\mathcal{D}_{\eta}(\mathcal{W})$ of visible \emph{complex dimensions} of $\eta$ is defined by 
\begin{equation}
\mathcal{D}_{\eta}(\mathcal{W}):=\{\omega\in\mathcal{W}:\,\zeta_{\eta}\mbox{\,has a pole at\,}\omega\}.
\end{equation}

\hspace*{3mm}For instance, the geometric zeta function of the M\"obius string (see Equation (\ref{GfMob})) is given by
\begin{equation}
\zeta_{\mu}(s)=\sum_{j=1}^{\infty}\frac{\mu(j)}{j^{s}}=\frac{1}{\zeta(s)},
\end{equation}
the reciprocal of the Riemann zeta function.\\

The \emph{spectral measure} $\nu$ associated to $\eta$ is defined by
\begin{equation}
\nu(A)=\sum_{k=1}^{\infty}\eta\big(\frac{A}{k}\big),
\end{equation} 
for any bounded Borel set $A\subset(0,+\infty)$.\,Then $N_{\nu}$ is called the \emph{spectral counting function} of $\eta$.\\

\hspace*{3mm}The \emph{spectral zeta function} $\zeta_{\nu}$ associated to $\eta$ is the geometric zeta function associated to $\nu$.\,It turn out (as shown in \cite{La-vF1, La-vF2, La-vF3, La-vF4}) that the spectral zeta function and the geometric zeta function of a generalized fractal string $\eta$ are related via the following formula:\footnote{In the special case of ordinary fractal strings, Equation (\ref{Geospec}) was first observed in \cite{La2} and used in \cite{La3, La4, LaPo1, LaPo2, LaMa1, LaMa2, La-vF1, La-vF2, La-vF3, La-vF4, Tep, La5, LalLa1, LalLa2, HerLa1, HerLa2, HerLa3}.\,Furthermore, a formula analogous to the one given in Equation (\ref{Geospec}) exists for other generalizations of ordinary fractal strings, including the class of fractal sprays (also called higher-dimensional fractal strings); see \cite{LaPo3, La2, La3, La-vF1, La-vF2, La-vF3, La-vF4}.}
\begin{equation}\label{Geospec}
\zeta_{\nu}(s)=\zeta_{\eta}(s).\zeta(s),
\end{equation}
where the factor $\zeta(s)$ is the Riemann zeta function.\\

\hspace*{3mm}Next, we introduce two generalized fractal strings which will play an important role in the definition of the spectral operator, its operator-valued prime factors and operator-valued Euler product, namely, the \emph{harmonic generalized fractal string}
\begin{equation}
\mathfrak{h}:=\sum_{k=1}^{\infty}\delta_{\{k\}},
\end{equation}
and for a fixed but arbitrary prime $p\in\mathcal{P}$, \emph{the prime harmonic string} 
\begin{equation}
\mathfrak{h}_{p}:=\sum_{k=1}^{\infty}\delta_{\{p^{k}\}}.
\end{equation}
These two generalized fractal strings are related via the multiplicative convolution of measures $\ast$ as follows (henceforth, $\mathcal{P}$ denotes the set of all prime numbers):
\begin{equation}
\mathfrak{h}=\underset{p\in\mathcal{P}}{\ast}\mathfrak{h}_{p}.
\end{equation}
As a result, we have 
\begin{equation}\label{Eq:Eph}
\zeta_{\mathfrak{h}}(s)=\zeta_{\underset{ p\in\mathcal{P}}{\ast\mathfrak{h_{p}}}}(s)=\zeta(s)=\underset{ p\in\mathcal{P}}{\prod}\frac{1}{1-p^{-s}}=\underset{ p\in\mathcal{P}}{\prod}\zeta_{\mathfrak{h}_{p}}(s),
\end{equation}
for $Re(s)>1$.\\

\hspace*{3mm}Given a generalized fractal string $\eta$, its spectral measure $\nu$ is related to  $\mathfrak{h}$ and $\mathfrak{h}_{p}$ via the following formula (where $\ast$ denotes the multiplicative convolution on $(0,+\infty)$):
\begin{equation}
\nu=\eta\ast\mathfrak{h}=\eta\ast\big(\underset{p\in\mathcal{P}}{\ast}\mathfrak{h}_{p}\big)=\underset{p\in\mathcal{P}}{\ast}\nu_{p},
\end{equation}
where $\nu_{p}=\eta\ast\mathfrak{h}_{p}$ is the spectral measure of $\mathfrak{h}_{p}$ for each $p\in\mathcal{P}$.\,Note that by applying the Mellin transform to the first equality of this identity, one recovers Equation (\ref{Geospec}).\\  

\hspace*{3mm}In their development of the theory of fractal strings in fractal geometry, the second author and M.\,van Frankenhuijsen obtained \emph{explicit distributional formulas}\footnote{Note that the original explicit formula was first obtained by Riemann in \cite{Rie} as an analytical tool aimed at understanding the distribution of the primes.\,We refer the reader to \cite{Edw, Ing, Ivi, Pat, Tit} for more details about Riemann's explicit formula and also to Appendix A in \cite{HerLa2} in which we provide a discussion of Riemann's explicit original formula and the explicit distributional formula obtained in [\textbf{La-vF4}, \S5.3 \& \S5.4].} associated to a given generalized fractal string $\eta$.\,These explicit formulas express the k-th distributional primitive (or anti-derivative) of $\eta$ (when viewed as a distribution) in terms of its complex dimensions.\,For simplicity and given the needs of our functional analytic framework, we will only present these formulas in a restricted setting and for the case of a strongly languid generalized fractal string.\footnote{Roughly speaking, a generalized fractal string $\eta$ is said to be \emph{strongly languid} if its geometric zeta function $\zeta_{\eta}$ satisfies some suitable polynomial growth conditions; see [\textbf{La-vF4}, \S5.3].}\,We encourage the curious reader to look at [\textbf{La-vF4}, \S5.3 \& \S5.4] for more details about the general statements of the explicit distributional formulas and a variety of their applications.\\

\hspace*{3mm}Given a strongly languid generalized fractal string $\eta$, and applying the explicit distributional formula at level $k=0$, we obtain an explicit representation of $\eta$, called the \emph{density of geometric states formula} (see [\textbf{La-vF4}, \S6.3.1]):\footnote{For simplicity, we assume here that all of the complex dimensions are simple poles of $\zeta_{\eta}$ and are different from 1.}
\begin{equation}\label{distr1}
\eta=\sum_{\,\omega\in\mathcal{D_{\eta}(W)\,}}res(\zeta_{\eta}(s);\omega)x^{\omega-1}.
\end{equation}
Applying the explicit distributional formula (at the same level $k=0$) to the spectral measure $\nu=\eta\ast\mathfrak{h}$, we obtain the following representation of $\nu$, called \emph{the density of spectral states formula} (or \emph{density of frequencies formula}) (see [\textbf{La-vF4}, \S6.3.1]):
\begin{equation}\label{distr2}
\nu=\zeta_{\eta}(1)+\sum_{\omega\in \mathcal{D}_{\eta}(W)}res(\zeta_{\eta}(s);\omega)\zeta(\omega)x^{\omega-1}.
\end{equation}

\begin{remark}
Many applications and extensions of fractal string theory and/or of the corresponding theory of complex fractal dimensions can be found throughout the books \cite{La-vF2, La-vF3, La-vF4, La-vF5, La5} and in \cite{La1, La2, La3, La4, LaPo1, LaPo2, LaPo3, LaMa1, LaMa2, HeLa, La-vF1, HamLa, Tep, LaPe, LaPeWi, LaLeRo, ElLaMaRo, LaLu1, LaLu2, LaLu-vF1, LaLu-vF2, LalLa1, LalLa2, LaRaZu, HerLa1, HerLa2, HerLa3, HerLa4, La6}.\,These include, in particular, applications to various aspects of number theory and arithmetic geometry, dynamical systems, spectral geometry, geometric measure theory, noncommutative geometry, mathematical physics and nonarchimedean analysis.
\end{remark}

\subsection{The spectral operator and the infinitesimal shifts $\partial_{c}$ of the real line}

\hspace*{3mm}The spectral operator was introduced by the second author and M.\,van Frankenhuijsen in [\textbf{La-vF3}, \S6.3.2] as a tool to investigate the relationship between the geometry and spectra of (generalized) fractal strings.\,While several aspects of their theory of complex dimensions in fractal geometry were in the process of being developed, they defined it `\emph{heuristically}' as the map that sends the geometry onto the spectrum of generalized fractal strings, but without providing the proper functional analytic framework needed to study it.\,(See [\textbf{La-vF3}, \S6.3.2] and [\textbf{La-vF4}, \S6.3.2].)\,A rigorous analytic framework enabling us to study the spectral operator was provided for the first time in \cite{HerLa1}.\,It was also surveyed in the papers \cite{HerLa2, HerLa3}.\,We will start by first defining the spectral operator and its operator-valued Euler product.\\

\hspace*{3mm}Given a generalized fractal string $\eta$, and in light of the distributional explicit formulas [\textbf{La-vF4}, Theorems 5.18 \& 5.22] (see Equations (\ref{distr1}) and (\ref{distr2}) above), the spectral operator $\mathfrak{a}$ was \emph{heuristically} defined as \emph{the operator mapping the density of geometric states of $\eta$ to its density of spectral states}:
\begin{equation}
\eta\mapsto \nu.
\end{equation}
Now, considering the level $k=1$ (in the explicit distributional formula),\footnote{That is, roughly speaking, take \textquotedblleft the antiderivative\textquotedblright \,of the corresponding expressions.} the spectral operator $\mathfrak{a}=\mathfrak{a}_{c}$ will be defined on a suitable weighted Hilbert space $\mathbb{H}_{c}$\footnote{The Hilbert space $\mathbb{H}_{c}$ will depend on a parameter $c\in\mathbb{R}$ which appears in the weight $\mu_{c}(t)=e^{-ct}dt$ defining $\mathbb{H}_{c}$ (see Equation (\ref{Hsp}) and the text surrounding it).\,As a result, the spectral operator will be denoted in the rest of the paper by either $\mathfrak{a}$ or $\mathfrak{a}_{c}$.} as \emph{the operator mapping the geometric counting function} $N_{\eta}$ \emph{onto the spectral counting function} $N_{\nu}$:
\begin{equation}
N_{\eta}(x)\longmapsto\nu(N_{\eta})(x):=N_{\nu}(x)=\sum_{n=1}^{\infty}N_{\eta}\left(\frac{x}{n}\right).
\end{equation}
Using the change of variable $x=e^{t}$, where $t\in\mathbb{R}$ and $x>0$, we obtain an \emph{additive} representation of the spectral operator $\mathfrak{a}$,
\begin{equation}\label{Eq:spop}
f(t)\mapsto \mathfrak{a}(f)(t)=\sum_{n=1}^{\infty}f(t-\log n).
\end{equation}
For each prime $p\in \mathcal{P}$, the \emph{operator-valued Euler factors} $\mathfrak{a}_{p}$ are given by
\begin{equation}
f(t)\mapsto \mathfrak{a}_{p}(f)(t)=\sum_{m=0}^{\infty}f(t-m\log p).
\end{equation}
All of these operators are related by an \emph{Euler product} as follows:

\begin{equation}
f(t)\mapsto \mathfrak{a}(f)(t)=\left(\prod_{p\in \mathcal{P}}\mathfrak{a}_{p}\right)(f)(t),\label{Eq:EPr}
\end{equation} 
where the product is the composition of operators.\\

\hspace*{3mm}Using the Taylor series representation of the function $f$,\footnote{Here, we will assume for pedagogical reasons that the complex-valued function $f$ is infinitely differentiable.\,Of course, this is not always the case for an arbitrary generalized fractal string (since the atoms of $\eta$ create discontinuities in the geometric counting function $f:=N_{\eta}$, for instance).\,This issue is addressed in \cite{HerLa1} by carefully studying the associated semigroup of operators; see Proposition \ref{prop:4.5} below.} 
\begin{equation}\label{Eq:Tfunc}
f(t+h)=e^{h\frac{d}{dt}}(f)(t)=e^{h\partial}(f)(t),
\end{equation}
where $\partial=\frac{d}{dt}$ is \emph{the infinitesimal shift of the real line} (or also the first order differential operator), we obtain the following representations of the spectral operator:
\begin{align}
\mathfrak{a}(f)(t)&=\sum_{n=1}^{\infty}e^{-(\log n)\partial}(f)(t)=\sum_{n=1}^{\infty}\left(\frac{1}{n^{\partial}}\right)(f)(t)\notag\\
                  &=\zeta(\partial)(f)(t)=\zeta_{\mathfrak{h}}(\partial)(f)(t).\label{Eq:Spr}
\end{align}
  
\hspace*{3mm}For all primes $p$, the operator-valued prime factors are given by
\begin{align}
\mathfrak{a}_{p}(f)(t)&=\sum_{m=0}^{\infty}f(t-m\log p)=\sum_{m=0}^{\infty}e^{-m(\log p)\partial}(f)(t)=\sum_{m=0}^{\infty}\left(p^{-\partial}\right)^{m}(f)(t)\notag\\
                      &=\left(\frac{1}{1-p^{-\partial}}\right)(f)(t)=(1-p^{-\partial})^{-1}(f)(t)=\zeta_{\mathfrak{h}_{p}}(\partial)(t),\label{Eq:Pf}
\end{align}
and hence, the operator-valued Euler product of the spectral operator $\mathfrak{a}$ is given by
\begin{equation}
\mathfrak{a}=\prod_{p\in \mathcal{P}}(1-p^{-\partial})^{-1}(f)(t)\label{Opvalprod}.
\end{equation}
\begin{remark}
Within our functional analytic framework, and in light of the above new representation of the spectral operator $\mathfrak{a}_{c}$, its operator-valued prime factors and operator-valued Euler product, the function $f$ will not necessarily represent the $($geometric or spectral$)$ counting function of some generalized fractal string $\eta$ but will instead be viewed as an element of the weighted Hilbert space $\mathbb{H}_{c}$ which was introduced in \cite{HerLa1} in order to study the spectral operator $($following, but also suitably modifying, a suggestion originally made in \cite{La-vF3}$)$.
\end{remark}
\begin{remark}
The above representation of the spectral operator given in Equation $($\ref{Eq:Spr}$)$ was justified in the functional analytic framework provided in \cite{HerLa1}.\\Furthermore, the representation of the operator-valued Euler product $\prod_{p\in\mathcal{P}}\mathfrak{a}_{p}$ $($see Equation $($\ref{Opvalprod}$)$$)$ and its operator-valued prime factors $\mathfrak{a}_{p}$  $($see Equation $($\ref{Eq:Pf}$)$$)$ were justified in \cite{HerLa4}.\,In particular, for $c>1$, the Euler product converges in the operator norm $($of the Hilbert space $\mathbb{H}_{c}$$)$.\,We note that the operator-valued Euler product associated to $\mathfrak{a}_{c}$ $($see Equation $($\ref{Opvalprod}$)$$)$ is conjectured to converge $($in an appropriate sense$)$ to $\mathfrak{a}_{c}$ in the critical strip; that is, for $0<c<1$.\,$($See $[$\emph{\textbf{La-vF3}}, \S6.3.3$]$ or $[$\emph{\textbf{La-vF4}}, \S6.3.3$]$.$)$\,In \cite{HerLa4}, a `\emph{suitable mode of convergence}' will be considered in order to address this conjecture.\,A full discussion of the operator-valued prime factors and operator-valued Euler product will be omitted in this paper.\,Instead, we will focus on studying some of the properties of the spectral operator $\mathfrak{a}_{c}$ and of its truncations $\mathfrak{a}^{(T)}_{c}$, as well as of the infinitesimal shift $\partial_{c}$ and its truncations $\partial_{c}^{(T)}$, which are the central characters of the present paper.\,Indeed, they will play a key role in our quantization of the universality of the Riemann zeta function, as will be discussed in \S5.\,$($See \S4 for a detailed discussion of these operators and some of their properties.$)$
\end{remark}

\hspace*{3mm}The spectral operator $\mathfrak{a}_{c}$ was rigorously defined in \cite{HerLa1} as an unbounded linear operator acting on the weighted Hilbert space
\begin{equation}
\mathbb{H}_{c}=L^{2}(\mathbb{R},\mu_{c}(dt)),\label{Hsp}
\end{equation}
where $c\geq 0$ and $\mu_{c}$ is the absolutely continuous measure on $\mathbb{R}$ given by $\mu_{c}(dt):=e^{-2ct}dt$, with $dt$ being the standard Lebesgue measure on $\mathbb{R}$.\,More precisely, $\mathfrak{a}_{c}$ is defined by
\begin{equation}
\mathfrak{a}_{c}(f)(t)=\zeta(\partial_{c})(f)(t).\label{RSpop}
\end{equation}
In view of Equation (\ref{RSpop}), the infinitesimal shift $\partial_{c}$ clearly plays an important role in our proposed definition of the spectral operator.\,The domain of $\partial_{c}$ is given by 
\begin{equation}
D(\partial_{c})=\{f\in \mathbb{H}_{c}\cap AC(\mathbb{R}):\,f'\in \mathbb{H}_{c}\},\label{Eq:acf}
\end{equation}
where $AC(\mathbb{R})$  is the space of (locally) absolutely continuous functions on $\mathbb{R}$ and $f'$ denotes the derivative of $f$, viewed either as an almost everywhere defined function or as a distribution.\footnote{For information about absolutely continuous functions and their use in Sobolev theory, see, e.g., \cite{Fo} and \cite{Br}; for the theory of distributions, see, e.g., \cite{Schw}, \cite{Ru} and \cite{Br}.}\,Furthermore, the domain of the spectral operator is given by
\begin{equation}
D(\mathfrak{a}_{c})=\{f\in D(\partial_{c}):\,\mathfrak{a}_{c}(f)=\zeta(\partial_{c})(f)\in\mathbb{H}_{c}\}.\label{dpart}
\end{equation}
Moreover, for $f\in D(\partial_{c})$, we have that
\begin{equation}
\partial_{c}(f)(t):=f'(t)=\frac{df}{dt}(t),
\end{equation}
almost everywhere (with respect to $dt$, or equivalently, with respect to $\mu_{c}(dt)$).\,In short, we write that $\partial_{c}(f)=f'$, where the equality holds in $\mathbb{H}_{c}$.\\

\hspace*{3mm}Finally, Equation (\ref{RSpop}) holds for all $f\in D(\mathfrak{a}_{c})$, and $\zeta(\partial_{c})$ is interpreted in the sense of the functional calculus for unbounded normal operators (see, e.g., \cite{Ru}).\,Indeed, as is recalled in Theorem \ref{Thm:normpartial} below, we show in \cite{HerLa1} that the operator $\mathfrak{a}_{c}$ is \emph{normal} $($i.e., that it is a closed, possibly unbounded operator, and that it commutes with its adjoint$)$.\,(It is bounded for $c>1$, but unbounded for $0<c\leq 1$, as is also shown in \cite{HerLa1, HerLa2, HerLa3}.)

\begin{remark}
The weighted Hilbert space $\mathbb{H}_{c}$ is the space of $($$\mathbb{C}$-valued$)$ Lebesgue square-integrable functions $f$ $($on $\mathbb{R}$$)$ with respect to the weight function $w(t)=e^{-2ct}$; namely, as was stated in Equation $($\ref{Hsp}$)$, $\mathbb{H}_{c}=L^{2}(\mathbb{R}, e^{-2ct}dt)$.\,It is equipped with the inner product
\begin{equation}
<f,\,g>_{c}:=\int_{\mathbb{R}}f(t)\overline{g(t)} e^{-2ct}dt,
\end{equation} 
where $\overline{g}$ denotes the complex conjugate of $g$.\,Therefore, the norm of an element of $\mathbb{H}_{c}$ is given by
\begin{equation}
||f||_{c}:=\bigg(\int_{\mathbb{R}}|f(t)|^{2}e^{-2ct}dt\bigg)^{\frac{1}{2}}.
\end{equation}
\end{remark}

\hspace*{3mm}It is shown in \cite{HerLa1} that the functions which lie in the domain of $\partial_{c}$ (and hence also those belonging to the domain of $\mathfrak{a}_{c}$, $\partial_{c}^{(T)}$ or $\mathfrak{a}_{c}^{(T)}$) satisfy `\emph{natural boundary conditions}'.\,Namely, if $f\in D(\partial_{c})$, then 
\begin{equation}
|f(t)|e^{-ct}\to 0 \mbox{\quad as\,\,}t\to\pm\infty.\label{SPopBdcond}
\end{equation}

\begin{remark}
The boundary conditions given in Equation $($\ref{SPopBdcond}$)$ are satisfied by any function $f$ in $D(\partial_{c})$ $($see Equation $($\ref{Eq:acf}$)$$)$ or in the domain of a function of $\partial_{c}$ such as the spectral operator $\mathfrak{a}_{c}=\zeta(\partial_{c})$ $($see Equation $($\ref{dpart}$)$$)$, as well as the truncated infinitesimal shifts $\partial_{c}^{(T)}$ $($see Equation $($\ref{Eq:truncshif}$)$$)$ and the truncated spectral operators $\mathfrak{a}_{c}^{(T)}$ $($see Equation $($\ref{Eq:ttsp}$)$$)$.\,Furthermore, clearly, if $\eta$ represents an ordinary fractal string $\mathcal{L}$, then $N_{\mathcal{L}}$ vanishes identically to the left of zero, and if, in addition, the $($Minkowski$)$ dimension of $\mathcal{L}$ is strictly less than $c$, then it follows from the results of \cite{LaPo2} that $N_{\eta}(t)=o(e^{ct})$ as $t\to +\infty$ $($i.e., $N_{\eta}(x)=o(x^{c})$ as $x\to +\infty$, in the original variable $x=e^{t}$$)$, so that $f:=N_{\eta}$ then satisfies the boundary conditions given by $($\ref{SPopBdcond}$)$.
\end{remark}

\begin{remark}
Note that it also follows from the results of \cite{LaPo2} $($see also \cite{La-vF4}$)$ that for $c$ in the `critical interval' $(0,1)$, the parameter $c$ can be interpreted as the least upper bound for the $($Minkowski$)$ fractal dimensions of the allowed underlying fractal strings.\footnote{For the notion of Minkowski dimension, see, e.g., \cite{Man, La1, Fa, Mat, La-vF4}.\,Recall that it was observed in \cite{La2} (using a result of \cite{BesTa}) that for an ordinary fractal string $\mathcal{L}$ (represented by a bounded open set $\Omega\subset \mathbb{R}$), the abscissa of convergence of $\zeta_{\mathcal{L}}$ coincides with the Minkowski dimension of $\mathcal{L}$ (i.e., of $\partial\Omega$); for a direct proof of this result, see [\textbf{La-vF4}, Theorem 1.10].}
\end{remark}

\hspace*{3mm}An intrinsic connection between the representations obtained in Equation (\ref{Eq:Spr}) was given in \cite{HerLa1}.\,Our next result justifies the representation of the spectral operator $\mathfrak{a}_{c}$ as the composition map of the Riemann zeta function and the infinitesimal shift $\partial_{c}$ (see Equation (\ref{Eq:Spr})):

\begin{theorem}
\,Assume that $c>1$.\,Then, $\mathfrak{a}$ can be uniquely extended to a bounded operator on $\mathbb{H}_{c}$ and, for any $f\in \mathbb{H}_{c}$, we have
\begin{equation}
\mathfrak{a}(f)(t)=\sum_{n=1}^{\infty}f(t-\log n)=\zeta(\partial_{c})(f)(t)=\left(\sum_{n=1}^{\infty}n^{-\partial_{c}}\right)(f)(t),\label{Eq:spcf}
\end{equation}
where the equalities hold for almost all $t\in \mathbb{R}$ as well as in $\mathbb{H}_{c}$.\label{Thm:11}
\end{theorem}

\begin{remark}
In other words, \emph{Theorem \ref{Thm:11}} justifies the `heuristic' representation of the spectral operator given above in Equation $($\ref{Eq:Spr}$)$.\,Indeed, it states that for $c>1$, we have
\begin{equation}
\mathfrak{a}_{c}=\zeta(\partial_{c})=\sum_{n=1}^{\infty}n^{-\partial_{c}},\label{Eq:2.23}
\end{equation}
where the equality holds in $\mathcal{B}(\mathbb{H}_{c})$, the space of bounded linear operators on $\mathbb{H}_{c}$.
\end{remark}

\begin{remark}
In addition, it is shown in \cite{HerLa1} that for any $c>0$, and for all $f$ in a suitable dense subspace of $D(\mathfrak{a}_{c})$ $($and hence, of $\mathbb{H}_{c}$$)$, an appropriate `analytic continuation' of Equation $($\ref{Eq:2.23}$)$ continues to hold $($when applied to $f$$)$.   
\end{remark}

\hspace*{3mm}A detailed study of the invertibility (and also, of the \emph{quasi-invertibility}\footnote{The spectral operator $\mathfrak{a}_{c}$ is said to be \emph{quasi-invertible} if its truncation $\mathfrak{a}^{(T)}$ is invertible for every $T>0$.\,See \cite{HerLa1, HerLa2} for a more detailed discussion of the quasi-invertibility of $\mathfrak{a}_{c}$, along with \S4.2 for an explanation of how the truncated spectral operators $\mathfrak{a}^{(T)}$, where $T\geq 0$, are defined.}) of the spectral operator $\mathfrak{a}_{c}$ is provided in \cite{HerLa1}.\,It was also surveyed in the papers \cite{HerLa2, HerLa3}.\,In particular, in that study, using the functional calculus along with the spectral mapping theorem for unbounded normal operators (the continuous version when $c\ne1$ and the meromorphic version, when $c=1$), a precise description of the spectrum $\sigma(\mathfrak{a}_{c})$  of the spectral operator is obtained in \cite{HerLa1}.\,More explicitly, we show that  $\sigma(\mathfrak{a}_{c})$ is equal to the closure of the range of the Riemann zeta function on the vertical line $L_{c}=\{\lambda\in\mathbb{C}:\,Re(\lambda)=c\}$:
\begin{theorem}\emph{\cite{HerLa1}}\,Assume that $c\geq0$.\,Then \label{Thm:ssop}
\begin{equation}\label{Eq:Merv}
\sigma(\mathfrak{a})=\overline{\zeta(\sigma(\partial))}=cl\big(\zeta(\{\lambda\in\mathbb{C}:\,Re(\lambda)=c\})\big),
\end{equation}
where $\sigma(\mathfrak{a})$\,is the spectrum of $\mathfrak{a}=\mathfrak{a}_{c}$ and $\overline{N}=cl(N)$ is the closure of $N\subset \mathbb{C}$.
\end{theorem}

In \cite{HerLa1} (see also \cite{HerLa3}), a spectral reformulation of the Riemann hypothesis is obtained, further extending from an operator-theoretic point of view the earlier reformulation of RH obtained by the second author and H.\,Maier in their study of the inverse spectral problem for fractal strings (see \cite{LaMa1, LaMa2}), in relation to answering the question (\`a la Mark Kac \cite{Kac}, but interpreted in a very different sense)
\begin{quotation}
\hspace*{14mm}\textquotedblleft Can one hear the shape of a fractal string?\textquotedblright.
\end{quotation} 

\begin{theorem}\emph{\cite{HerLa1}}\label{Thm:qual}
The spectral operator $\mathfrak{a}=\mathfrak{a}_{c}$ is quasi-invertible for all $c\in(0,1)-\frac{1}{2}$ $($or equivalently, for all $c\in (\frac{1}{2},1)$$)$ if and only if the Riemann hypothesis is true.
\end{theorem}

\begin{remark}
The above theorem also enables us to give a proper formulation of $[$\emph{\textbf{La-vF4}}, Corollary 9.6$]$ in terms of the quasi-invertibility of the spectral operator $\mathfrak{a}_{c}$, as interpreted in \cite{HerLa1, HerLa2, HerLa3}.
\end{remark}

\hspace*{3mm}In the next subsection, we will discuss some of the fundamental properties (obtained in \cite{HerLa1, HerLa2}) of the infinitesimal shift $\partial=\partial_{c}$ and the associated strongly continuous group $\{e^{-t\partial}\}_{t\geq 0}$.

\subsection{Properties of the infinitesimal shifts $\partial_{c}$}
The infinitesimal shift $\partial=\partial_{c}$, with $c\geq 0$ (along with their truncations $\partial^{(T)}=\partial_{c}^{(T)}$, $T\geq0$, to be discussed in \S4.1) are the basic building blocks of the theory developed in \cite{HerLa1, HerLa2, HerLa3, HerLa4} as well as in the present paper, as we shall see, in particular, in \S5.

\begin{theorem}\emph{\cite{HerLa1}}\label{Thm:normpartial}
The infinitesimal shift $\partial=\partial_{c}$ is an unbounded normal linear operator on $\mathbb{H}_{c}$.\,Moreover, its adjoint $A^{*}$ is given by
\begin{equation}
\partial^{*}=2c-\partial, \mbox{\quad \emph{with}\quad} D(\partial^{*})=D(\partial).\label{Eq:dp}
\end{equation}
\end{theorem}

\begin{remark}\label{Cor:Resspect}
The residual spectrum, $\sigma_{r}(A)$, of a normal $($possibly unboun-\\ded$)$ linear operator $A:D(A)\subset\mathcal{H}\to\mathcal{H}$, where $D(A)$ is the domain of $A$ and $\mathcal{H}$ is some complex Hilbert space, is empty.\,Hence, the essential spectrum of $A$, $\sigma_{e}(A)$, which consists of all the \emph{approximate eigenvalues} of $A$,\footnote{Recall that $\lambda\in\mathbb{C}$ is called an \emph{approximate eigenvalue} of $A$ if there exists a sequence of unit vectors $\{\psi_{n}\}_{n=1}^{\infty}$ of $\mathcal{H}$ such that $(A-\lambda)\psi_{n}\to 0$ as $n\to \infty$.\,(See, e.g., \cite{Sc}.)} is equal to the entire spectrum of $A$, denoted by $\sigma(A)$.
\end{remark}

\begin{remark}
References on the spectral theory $($and the associated functional calculus$)$ of unbounded linear operators, with various degrees of generality and emphasis on the applications of the theory, include \cite{DunSch, Kat, Ru, ReSi, Sc, JoLa, Ha}.\,In particular, the case of unbounded normal operators $($which is of most direct interest here$)$ is treated in Rudin's book \cite{Ru}.
\end{remark}

\begin{theorem}\emph{\cite{HerLa1}}\label{Thm:spectrum(partial)}
The spectrum, $\sigma(\partial)$, of the differentiation operator $($or infinitesimal shift$)$ $\partial=\partial_{c}$ is equal to the closed vertical line of the complex plane passing through $c\geq 0$; furthermore, it coincides with the essential spectrum of $\partial$$:$
\begin{equation}
\sigma(\partial)=\sigma_{e}(\partial)=\{\,\lambda\in \mathbb{C}:\,Re(\lambda)=c\,\},\label{Thm:spect}
\end{equation}
where $\sigma_{e}(\partial)$ consists of all the approximate eigenvalues of $\partial$.\,Moreover, the point spectrum of $\partial$ is empty $($i.e., $\partial$ does not have any eigenvalues$)$, so that the spectrum of $\partial$ is purely continuous. 
\end{theorem}

\begin{corollary}\label{Cor:sigma}
For any $c\geq 0$, we have $\sigma(\partial^{*})=\sigma(\partial)=c+i\mathbb{R}$.
\end{corollary}

\hspace*{3mm}In light of Corollary \ref{Cor:sigma}, the following result is really a consequence of Theorem \ref{Thm:normpartial} and will be very useful to us in the sequel (see \S4 and \S5):

\begin{corollary}\label{Cor:Acv}
For any $c\geq0$, we can write 
\begin{equation}
\partial=c+iV \mbox{\quad \emph{and}\quad}\,\partial^{*}=c-iV,
\end{equation}
where $c=cI=Re(\partial)$ $($a constant multiple of the identity operator on $D(\partial)\subset\mathbb{H}_{c}$$)$ and $V=Im(\partial)$, an unbounded self-adjoint operator on $\mathbb{H}_{c}$, with domain $D(V)=D(\partial)=D(\partial^{*})$.\,Moreover, the spectrum of $V$ is given by $\sigma(V)=\mathbb{R}$.
\end{corollary}

\hspace*{3mm}The next result (also from \cite{HerLa1}) enables us to justify the use of the term \textquotedblleft infinitesimal shift\textquotedblright (when referring to the operator $\partial=\partial_{c}$) as well as some of the formal manipulations occurring in Equations (\ref{Eq:Tfunc}), (\ref{Eq:Spr}), (\ref{Eq:Pf}) and (\ref{Opvalprod}) of \S3.2 above.\footnote{Detailed information about the theory of semigroups of bounded linear operators can be found, e.g., in the books \cite{HiPh, EnNa, Paz, Kat, ReSi, JoLa}.}

\begin{proposition}\label{prop:4.5}
Fix $c\geq0$ and write $\partial=\partial_{c}$.\,Then, the following two properties hold\emph{:}\\

$($i$)$\,$\{e^{-t\partial}\}_{t\geq 0}$ is a strongly continuous contraction semigroup of bounded linear operators on $\mathbb{H}_{c}$ and $||e^{-t\partial}||=e^{-tc}$ for any $t\geq0$.\,Hence, its infinitesimal generator $\partial$ is an $m$-accretive operator on $\mathbb{H}_{c}$ $($in the sense of \emph{\cite{Kat, JoLa, Paz}}$)$.\\

$($ii$)$$\,\{e^{-t\partial}\}_{t\geq 0}$ is a translation $($or shift$)$ semigroup.\,That is, for every $t\geq0$, $(e^{-t\partial})(f)(u)=f(u-t)$, for all $f\in \mathbb{H}_{c}$ and $u\in \mathbb{R}$.\,$($For a fixed $t\geq 0$, this equality holds between elements of $\mathbb{H}_{c}$ and hence, for a.e. $u\in \mathbb{R}$.$)$
\end{proposition}

\begin{remark}
An entirely analogous result holds for the semigroup $\{e^{t\partial}\}_{t\geq 0}$, except that it is then an expanding $($rather than a contraction$)$ semigroup.\,Similarly, for any $c\in\mathbb{R}$ such that $c\leq 0$, all of the results stated in \S3.3 are still valid without changes except for the fact that in Proposition \ref{prop:4.5}, the adjectives \textquotedblleft contractive\textquotedblright \,and  \textquotedblleft expanding\textquotedblright  \,must be interchanged.
\end{remark}

\section{The Truncated Spectral Operators $\mathfrak{a}_{c}^{(T)}=\zeta(\partial_{c}^{(T)})$}

\hspace*{3mm}As was alluded to above (at the beginning of \S3.3), a quantum analog of the universality of the Riemann zeta function (to be provided in \S5 below) will be expressed (in our context) in terms of the truncated infinitesimal shifts (to be defined in \S4.1 below) and also, in some sense, in terms of the truncated spectral operators (to be defined in \S4.2 below). 

\subsection{The truncated infinitesimal shifts $\partial_{c}^{(T)}$ and their properties}

\hspace*{3mm}Recall from Corollary \ref{Cor:Acv} that for any $c\geq 0$, the infinitesimal shift $\partial=\partial_{c}$ is given by
\begin{equation}
\partial=c+iV,\notag
\end{equation}
where $V:=Im(\partial)$ (the imaginary part of $\partial$) is an unbounded self-adjoint operator such that $\sigma(V)=\mathbb{R}$.\\

\hspace*{3mm}Given $T\geq 0$, we define the $T$-\emph{truncated infinitesimal shift} as follows:
\begin{equation}\label{Eq:truncshif}
\partial^{(T)}:=c+iV^{(T)},
\end{equation}
where
\begin{equation}\label{Eq:VT}
V^{(T)}:=\phi^{(T)}(V)
\end{equation}
(in the sense of the functional calculus), and $\phi^{(T)}$ is a suitable (i.e., $T$-admissible) continuous (if $c\ne1$) or meromorphic (if $c=1$) cut-off function chosen so that $\overline{\phi^{(T)}(\mathbb{R})}=c+i[-T,T]$).\\

\hspace*{3mm}The next result states that the spectrum $\sigma(\partial_{c}^{(T)})$ of the truncated infinitesimal shift is equal to the vertical line segment of height $T$ and abscissa $c$, symetrically located with respect to $c$.\\

\begin{theorem}\emph{\cite{HerLa1}}\label{Thm:Truncv}
For any $T>0$ $($and $c\geq0$$)$, the spectrum $\sigma(\partial^{(T)})$ of the truncated infinitesimal shift $\partial^{(T)}=\partial^{(T)}_{c}$ is given by 

\begin{equation}\label{Eq4.7}
\sigma(\partial^{(T)})=\{c+i\tau:\,|\tau|\leq T,\,\tau\in\mathbb{R}\}=c+i[-T,T].
\end{equation}

\hspace*{3mm}Moreover, the spectrum $\sigma(V^{(T)})$ of the imaginary part $V^{(T)}$ of the infinitesimal shift is given by 
\begin{equation}\label{Eq4.8}
\sigma(V^{(T)})=[-T,T].
\end{equation}
\end{theorem}

\subsection{The truncated spectral operators and their spectra}

\hspace*{3mm}Next, we define our other main objects of study, the truncated spectral operators (denoted by $\mathfrak{a}^{(T)}$) which, along with the truncated infinitesimal shifts $\partial_{c}^{(T)}$ introduced in \S4.1 just above, will play a crucial role in our proposed operator-theoretic extension of the universality theorem for the Riemann zeta function (see Theorems \ref{Thm:com} and \ref{Extendoperatorval}).\,Let $c\geq 0$.\,Then, given $T\geq 0$, the $T$-\emph{truncated spectral operator} is defined as follows:

\begin{equation}\label{Eq:ttsp}
\mathfrak{a}^{(T)}:=\zeta\left(\partial^{(T)}\right).
\end{equation}

\hspace*{3mm}More precisely, in the definition of $\partial^{(T)}=c+iV^{(T)}$, with $V^{(T)}=\phi^{(T)}(V)$, as given in Equation (\ref{Eq:truncshif}) and Equation (\ref{Eq:VT}), the $T$-admissible function $\phi^{(T)}$ is chosen as follows:\\

\vspace*{3mm}
\hspace*{3mm}(i)\hspace*{2mm}If $c\ne 1$, $\phi^{(T)}$ is any continuous function such that $\overline{\phi^{(T)}}(\mathbb{R})=[-T,T]$.\,(For example, $\phi^{(T)}(\tau)=\tau$ for $0\leq \tau\leq T$ and $\phi^{(T)}(\tau)=T$ for $\tau\geq T$; also, $\phi^{(T)}$ is odd.)\\

\vspace*{2mm}
\hspace*{3mm}(ii)\hspace*{2mm}If $c=1$ (which corresponds to the pole of $\zeta(s)$ at $s=1$), then $\phi^{(T)}$ is a suitable meromorphic analog  of (i).\,(For example, $\phi^{(T)}(s)=\frac{2T}{\pi}\tan^{-1}(s)$, so that $\overline{\phi^{(T)}(\mathbb{R})}=[-T,T]$.)

\vspace*{3mm}
\hspace*{3mm} One then uses the measurable functional calculus and an appropriate (continuous or meromorphic, for $c\ne1$ or $c=1$, respectively) version of the spectral mapping theorem (\textbf{SMT}) for unbounded normal operators (as provided in [\textbf{HerLa1}, Appendix E]) in order to define both $\partial^{(T)}$ and $\mathfrak{a}^{(T)}=\zeta(\partial^{(T)})$, as well as to determine their spectra (see \S4.1  above for the case of $\partial^{(T)}$):
\[
\textbf{SMT}:\hspace{1mm}\sigma(\psi(L))=\overline{\psi(\sigma(L)})
\]
\emph{if $\psi$ is a continuous $($resp., meromorphic$)$ function on $\sigma(L)$ $($resp., on a connected open neighborhood of $\sigma(L)$$)$ and $L$ is an unbounded normal operator.}

\begin{remark}\label{Rk4.2}
More precisely, in the meromorphic case, in the above equality $($in the statement of \emph{\textbf{SMT}}$)$, one should exclude the poles of $\psi$ which belong to $\sigma(L)$.\,Alternatively, one can view the meromorphic function $\psi$ as a continuous function with values in the Riemann sphere $\widetilde{\mathbb{C}}:=\mathbb{C}\cup \{\infty\}$ and then write \textbf{SMT} in the following simpler form\emph{:}
\begin{equation}
\sigma(\psi(L))=\psi(\sigma(L)).\notag
\end{equation} 
$($See $[$\emph{\textbf{HerLa1}, Appendix E}$]$, along with the relevant references therein, including \cite{Ha}.$)$
\end{remark}

\hspace*{3mm}Note that for $c\ne 1$ (resp., $c=1$), $\partial^{(T)}$ and $\mathfrak{a}^{(T)}$ are then continuous (resp., meromorphic) functions of the normal (and sectorial, see \cite{Ha}) operator $\partial$.\,An entirely analogous statement is true for the spectral operator $\mathfrak{a}=\zeta(\partial)$.

\begin{theorem}\emph{\cite{HerLa1}}\label{Thm:mfrb}
\hspace*{3mm}$($i$)$\,Assume that $c\geq0$, with $c\ne1$.\,Then, for all $T\geq 0$, $\mathfrak{a}^{(T)}$ is a bounded normal linear operator.\,Furthermore, its spectrum $\sigma(\mathfrak{a}^{(T)})$ is given by the following compact $($and hence, bounded$)$ subset of $\mathbb{C}:$
\begin{equation}\label{Spcmfrk}
\sigma(\mathfrak{a}^{(T)})=\{\zeta(c+i\tau):\,|\tau| \leq T, \tau\in\mathbb{R}, \tau\ne 0 \}.
\end{equation}

$($ii$)$\,When $c=1$, a similar statement holds for all $T>0$ except that now, $\mathfrak{a}^{(T)}$ is an unbounded $($i.e., not bounded$)$ normal operator with spectrum given by $($with $cl$ denoting the closure of a set$)$
\begin{equation}\label{Eq:sigmfrk}
\sigma(\mathfrak{a}^{(T)})=cl \{\zeta(1+i\tau):\,|\tau|\leq T, \tau\in \mathbb{R}\},\mbox{\quad} 
\end{equation}
a non-compact $($and in fact, unbounded$)$ subset of $\mathbb{C}$.\,Alternately, one could write 
\begin{equation}\label{Eq:unbeq}
\widetilde{\sigma}(\mathfrak{a}^{(T)})=\{\zeta(1+i\tau):\,|\tau|\leq T,\,\tau\in \mathbb{R}\},
\end{equation}
a compact subset of the Riemann sphere $\widetilde{\mathbb{C}}=\mathbb{C}\cup\{\infty\}$, where $\zeta$ is viewed as a $($continous$)$ $\widetilde{\mathbb{C}}$-valued function and the extended spectrum $\widetilde{\sigma}(\mathfrak{a}^{(T)})$ of $\mathfrak{a}^{(T)}$ is given by $\widetilde{\sigma}(\mathfrak{a}^{(T)}):=\sigma(\mathfrak{a}^{(T)})\cup \{\infty\}$ $($still when $c=1$$)$.\footnote{When $c\ne1$ (i.e., in case (i)), we have that $\widetilde{\sigma}(\mathfrak{a}^{(T)})=\sigma(\mathfrak{a}^{(T)})$ since $\mathfrak{a}^{(T)}$ is then bounded; see, e.g., \cite{Ha} for the notion of extended spectrum.\,In short, $\widetilde{\sigma}(L):=\sigma(L)$ if the linear operator $L$ is bounded, and $\widetilde{\sigma}(L):=\sigma(L)\cup \{\infty\}$ if $L$ is unbounded; so that $\widetilde{\sigma}(L)$ is always a compact subset of $\widetilde{\mathbb{C}}$.}
\end{theorem}

\begin{proof}
We include this particular proof in order to illustrate the use of the spectral mapping theorem \textbf{SMT} discussed in Remark \ref{Rk4.2} and the text preceding it.\\

\hspace*{3mm}In light of Equation (\ref{Eq:ttsp}), this follows immediately from the continuous ($c\ne1$) or meromorphic ($c=1$) version of the spectral mapping theorem,\footnote{This theorem is applied to the function $\zeta$ and the bounded normal operator $\partial^{(T)}$ studied in \S4.1.} according to which 
\begin{equation}\label{Eq:sigzet}
\sigma(\mathfrak{a}^{(T)})=\sigma(\zeta(\partial^{(T)}))=cl\{\zeta(\sigma(\partial^{(T)}))\}.
\end{equation}
Now, it follows from the results of \S4.1 (see Theorem \ref{Thm:Truncv}) that
\begin{equation}\label{Eq:sigpar}
\sigma(\partial^{(T)})=\{c+i\tau:\,|\tau|\leq T, \tau\in\mathbb{R}\}.
\end{equation}
Therefore, combining Equations (\ref{Eq:sigzet}) and (\ref{Eq:sigpar}), we obtain
\begin{equation}\label{Eq:sigmfrkzet}
\sigma(\mathfrak{a}^{(T)})=cl(\zeta(c+i\tau:|\tau|\leq T, \tau\in\mathbb{R}\}).
\end{equation}

\hspace*{3mm}Note that if $c\ne1$, then $\zeta$ is continuous on the vertical line $Re(s)=c$ and hence, on the compact vertical segment given by Equation $($\ref{Eq:sigpar}$)$.\,Hence, its range along this segment is compact and therefore closed in $\mathbb{C}$.\,This explains why we do not need to include the closure in Equation (\ref{Spcmfrk}) giving the expression of the spectrum $\sigma(\mathfrak{a}^{(T)})$ of $\mathfrak{a}^{(T)}$ when $c\ne 1$.\,Moreover, since $\sigma(\mathfrak{a}^{(T)})$ is compact and hence, bounded, the truncated infinitesimal shift $\mathfrak{a}^{(T)}$ is a bounded operator.\\

\hspace*{3mm}On the other hand, when $c=1$, $\zeta$ is meromorphic in an open neighborhood of the vertical line $\{Re(s)=c=1\}$, and has a (simple) pole at $s=1$.\,It follows that its range is unbounded along the vertical segment 
\begin{equation}
\sigma(\partial^{(T)})=\{1+i\tau:\,|\tau|\leq T, \tau\in \mathbb{R}, \tau\ne 0\}.
\end{equation} 
Therefore, still when $c=1$, we must keep the closure in the expression for $\sigma(\mathfrak{a}^{(T)})$ given by Equation $($\ref{Eq:sigmfrkzet}$)$.\,In addition, as stated in part $($ii$)$ of the theorem, $\mathfrak{a}^{(T)}$ is an unbounded normal operator when $c=1$ because its spectrum is unbounded (and thus, non-compact).\footnote{Note that since the spectrum is always a closed subset of $\mathbb{C}$, it is non-compact if and only if the operator is unbounded.\,On the other hand, when $c=1$, the extended spectrum of $\mathfrak{a}^{(T)}$ is defined by $\widetilde{\sigma}(\mathfrak{a}^{(T)})=\sigma(\mathfrak{a}^{(T)})\cup\{\infty\}$ $($since the operator $\mathfrak{a}^{(T)}$ is unbounded, see, e.g., \cite{Ha}$)$ and is a closed $($and hence, compact$)$
subset of the Riemann sphere $\widetilde{\mathbb{C}}$.\,It is therefore still given by the right-hand side of Equation (\ref{Eq:unbeq}), but with $\zeta$ viewed as a continuous function with values in $\widetilde{\mathbb{C}}$, as is explained in Remark \ref{Rk4.2}.} 

\end{proof}

\section{Truncated Infinitesimal Shifts and Quantum Universality of $\zeta(s)$}

\hspace*{3mm}In the present section, we provide a `quantum' analog of the universality of the Riemann zeta function $\zeta=\zeta(s)$.\,Somewhat surprisingly at first, in our context, the proper formulation of Voronin's universality theorem (and its various generalizations) does not simply consists in replacing the complex variable $s$ with the infinitesimal shift $\partial=\partial_{c}$.\,In fact, as we shall see, one must replace the complex variable $s$ with the family of truncated infinitesimal shifts $\partial^{(T)}=\partial_{c}^{(T)}$ (with $T\geq 0$ and $c\geq 0$).\,Therefore, strictly speaking, it would not be correct to say that the spectral operator $\mathfrak{a}=\zeta(\partial)$ is universal (and since $\zeta$ is a highly nonlinear function, we could not say either that the family of truncated spectral operators $\mathfrak{a}^{(T)}=\zeta(\partial^{(T)})$ is universal).\,Instead, the proper statement of `quantum universality' is directly expressed in terms of the truncated infinitesimal shifts $\partial^{(T)}$ and their imaginary translates.

\subsection{An operator-valued extension of Voronin's theorem}
\hspace*{3mm}The \textquotedblleft universality\textquotedblright of the spectral operator $\mathfrak{a}=\zeta(\partial)$ roughly means that any non-vanishing holomorphic function of $\partial$ on a suitable compact subset of the right critical strip $\{\frac{1}{2}<Re(s)<1\}$ can be approximated (in the operator norm) arbitrarily closely by imaginary translates of $\zeta(\partial)$.\,More accurately, any  such non-vanishing function of the truncated infinitesimal shifts $\partial^{(T)}=\partial_{c}^{(T)}$ can be uniformly (in the parameters $c$ and $T$) approximated (in the operator norm on $\mathbb{H}_{c}$) by the composition of $\zeta$ and suitable imaginary translates of $\partial^{(T)}=\partial_{c}^{(T)}$.\\  

\hspace*{3mm}Indeed, we have the following \emph{operator-theoretic generalization of the extended Voronin universality theorem}, expressed in terms of the imaginary translates of the $T$-\emph{truncated infinitesimal shifts} $\partial^{(T)}=\partial_{c}^{(T)}$ (with parameter $c$).\\

\hspace*{3mm}We begin by providing an operator-theoretic generalization of the universality theorem which is in the spirit of Voronin's original universality theorem (Theorem \ref{VorThm}) and its extension (Theorem \ref{Thm:ExtVor}).\\

\begin{theorem}\cite{HerLa1}\,\label{Thm:com}$($\emph{Quantized universality of} $\zeta(s)$\emph{; first version}$)$.\,
\hspace*{3mm}Let K be a compact subset of the right critical strip $\{\frac{1}{2}<Re(s)<1\}$ of the following form.\,Assume, for simplicity, that $K=\mathcal{K}\times[-T_{0},T_{0}]$, for some $T_{0}\geq0$, where $\mathcal{K}$ is a compact subset of the open interval $(\frac{1}{2},1)$.\\
\hspace*{5mm}Let $g:K\to\mathbb{C}$ be a non-vanishing $($i.e., nowhere vanishing$)$ continuous function that is holomorphic in $\mathring{K}$, the interior of $K$ $($which may be empty$)$.\,Then, given any $\epsilon>0$, there exists $\tau\geq 0$ $($depending only on $\epsilon$$)$ such that
\begin{equation}
\mathcal{H}_{op}(\tau):=\underset{c\in\mathcal{K},\,0<T\leq T_{0}}{\sup}\Big|\Big|g(\partial_{c}^{(T)})-\zeta(\partial_{c}^{(T)}+i\tau)\Big|\Big|\leq\epsilon,
\end{equation}
where $\partial^{(T)}=\partial_{c}^{(T)}$ is the $T$-truncated infinitesimal shift $($with parameter $c$$)$ and $||.||$ is the norm in $\mathcal{B}(\mathbb{H}_{c})$ $($the space of bounded linear operators on $\mathbb{H}_{c}$$)$.
\end{theorem}

\hspace*{3mm}\emph{Moreover, the set of all such $\tau$'s has a positive lower density and, in particular, is infinite.\,More precisely, we have}
\begin{equation}\label{Eq:ldens}
\underset{\rho\to+\infty}{\liminf}\,\frac{1}{\rho}\,vol_{1}\left(\{\tau\in[0,\rho]:\,\mathcal{H}_{op}(\tau)\leq\epsilon\}\right)>0.
\end{equation}

\begin{proof}\,For $\tau>0$, let $K=\mathcal{K}\times [-T,T]$.\,Then, we consider the following two cases:\\

(i)\,\,If $T_{0}=0$, then $\mathring{K}=\emptyset$ (interior in $\mathbb{C}$).\,Also, if $\mathring{\mathcal{K}}=\emptyset$ (interior in $\mathbb{R}$), then $\mathring{K}=\emptyset$ (interior in $\mathbb{C}$).\,In either case, we only need to know that $g$ is continuous on $\mathcal{K}\times [-T_{0},T_{0}]\subseteq\mathcal{K}\times \mathbb{R}$.\,The remainder of the proof, however, proceeds exactly as in part (ii) below, by applying Theorem \ref{Thm:ExtVor} to the continuous function $g$ and the compact set $K=\mathcal{K}\times [-T_{0},\,T_{0}]$ with empty interior in $\mathbb{C}$ (as well as to the restriction of $g$ to $K_{T}:=\mathcal{K}\times [-T,T]$, with $T$ such that $0<T\leq T_{0}$).\\

(ii)\,\,If $T_{0}>0$ and $\mathring{\mathcal{K}}\ne\emptyset$, then $\mathring{K}\ne \emptyset$.\,Then, we need to require that $g$ is holomorphic in the interior of $K=\mathcal{K}\times [-T_{0}, T_{0}]$, in addition to being continuous on $K$.\,Now, by the universality of the Riemann zeta function applied to the nowhere vanishing function $g:K\to \mathbb{C}$ (or, more specifically, by applying Theorem \ref{Thm:ExtVor} to $g:K\to\mathbb{C}$, where $K=\mathcal{K}\times [-T_{0}, T_{0}]$),\footnote{Note that the complement of $K$ in $\mathbb{C}$ is connected (since the complement of $\mathcal{K}$ in $\mathbb{R}$, being an open subset of $\mathbb{R}$, is an at most countable disjoint union of intervals).} we conclude that given $\epsilon>0$ there exists $\tau=\tau(T_{0},\epsilon)\geq0$ such that 
\begin{equation}\label{Eq:8.2}
|g(s)-\zeta(s+i\tau)|\leq \epsilon, \mbox{\quad for all\,\,} s\in K.
\end{equation}   
In addition, the set of such numbers $\tau$ has positive lower density.\\

\hspace*{3mm}Next, given $T$ such that $0<T\leq T_{0}$, let us set $K_{T}=\mathcal{K}\times [-T,T]$, so that $K=K_{T_{0}}$.\,Clearly, for any such $T$, we have $K_{T}\subseteq K_{T_{0}}$.\,Hence, since it follows from Theorem \ref{Thm:Truncv} in \S4.1 (where the spectrum of $\partial^{(T)}$ is precisely determined), that for every $c\in\mathcal{K}$, we have\footnote{Recall that $\mathcal{K}\subset (\frac{1}{2},1)$ and that $\partial^{(T)}=\partial_{c}^{(T)}$, so that $\zeta(\partial_{c}^{(T)}+i\tau)$ is a bounded (normal) operator.\,Similarly, $\phi$ is continuous on the compact set $\sigma(\partial_{c}^{(T)})=[c-iT, c+iT]$, and hence $\phi(\partial_{c}^{(T)})$ is a bounded (normal) operator.\,In any case, for every $c\in\mathcal{K}$, $g$ is continuous and is therefore bounded on the compact set $\sigma(\partial_{c}^{(T)})$, in agreement with Equation (\ref{Eq:8.3}).\,Furthermore, $g(\partial_{c}^{(T)})$ is a bounded (normal) operator and its norm satisfies an inequality implied by (\ref{Eq:8.4}).}
\begin{align}
\sigma(\partial^{(T)})
&=[c-iT, c+iT]=\{c\}\times [-T,T]\notag\\
&\subseteq \sigma(\partial^{(T_{0})})=[c-iT_{0}, c+iT_{0}]=\{c\}\times [-T_{0}, T_{0}]\notag\\
&\subseteq K=K_{T_{0}}=\mathcal{K}\times [-T_{0}, T_{0}].
\end{align}
We conclude that $g$ is continuous (and thus certainly measurable) on $\sigma(\partial^{(T)})$.\\

\hspace*{3mm}We can therefore apply to $\partial^{(T)}=\partial_{c}^{(T)}$ the continuous version of the functional calculus (for unbounded normal operators) to deduce that 
\begin{equation}\label{Eq:8.5}
\phi(\partial_{c}^{(T)})=g(\partial_{c}^{(T)})-\zeta(\partial_{c}^{(T)}+i\tau),
\end{equation}
where $\phi(s):=g(s)-\zeta(s+i\tau)$, for $s\in K$; so that, in light of Equation (\ref{Eq:8.2}),
\begin{equation}\label{Eq:8.3}
|\phi(s)|\leq \epsilon, \mbox{\quad for all\,\,} s\in K.
\end{equation}
(Recall that each operator $\partial^{(T)}=\partial_{c}^{(T)}$ is bounded on $\mathbb{H}_{c}$ since it has a bounded spectrum.)\,This same functional calculus (or, equivalently, the corresponding version of the spectral theorem for possibly unbounded normal operators, see \cite{Ru}) implies that each of the operators $g(\partial_{c}^{(T)})$, $\zeta(\partial_{c}^{(T)}+i\tau)$ and $\phi(\partial_{c}^{(T)})$ belongs to $\mathcal{B}(\mathbb{H}_{c})$,\footnote{We use here, in particular, the fact that $\zeta$ is continuous on the compact vertical segment $\sigma(\partial_{c}^{(T)}+i\tau)=\sigma(\partial_{c}^{(T)})+i\tau=[c+i(\tau-T), c+i(\tau+T)]$ because $c\ne 1$ for $c\in\mathcal{K}$.} and 
\begin{equation}\label{Eq:8.4}
||\phi(\partial_{c}^{(T)})||=\sup_{s\in\sigma(\partial_{c}^{(T)})}|\phi(s)|\leq \sup_{s\in K}|\phi(s)|\leq \epsilon.
\end{equation}
Note that the first inequality follows from the fact that $\sigma(\partial_{c}^{(T)})\subseteq K$ while the second inequality follows from Equation (\ref{Eq:8.3}).\,Since Equation (\ref{Eq:8.4}) holds for every $c\in \mathcal{K}$ and for every $T$ such that $0<T\leq T_{0}$, and recalling that $\phi(\partial_{c}^{(T)})$ is given by the identity (\ref{Eq:8.5}) (which holds in $\mathcal{B}(\mathbb{H}_{c})$), we conclude that given any $\epsilon>0$ (and for a fixed $T_{0}\geq0$), there exists $\tau=\tau(\epsilon)\geq 0$ such that  

\begin{equation}
\sup_{c\in \mathcal{K},\,0<T\leq T_{0}}||g(\partial_{c}^{(T)})-\zeta(\partial_{c}+i\tau)||\leq \epsilon,\notag
\end{equation}
as desired.\,Furthermore, the set of such $\tau$'s has positive lower density; i.e., inequality (\ref{Eq:ldens}) holds.
\end{proof}

\begin{remark}
A remarkable feature of the above generalization is the uniformity in the parameter $c\in \mathcal{K}$ and in $T\in [0,T_{0}]$ of the stated approximation of $g(\partial_{c}^{(T)})$.
\end{remark}

\subsection{A more general operator-valued extension of Voronin's theorem}

\hspace*{3mm}We will next state a further generalization of the \emph{operator-theoretic extended Voronin universality theorem}.\,For pedagogical reasons, we will choose assumptions (on the compact set $K$) that simplify its formulation.\,We begin by introducing the notion of `vertical convexity' of a given set, which will be used in Theorem \ref{Extendoperatorval} below.

\begin{definition}\label{Rk:Vcon}
To say that $K$ is \emph{vertically convex} means that if $c-iT'$ and $c+iT$ belong to $K$ for some $c\in K$ and $T'\leq 0\leq T$, then the entire vertical line segment $[c-iT',c+iT]$ is contained in $K$.\\
\end{definition}

\begin{theorem}\cite{HerLa1}\,$($\emph{Quantized universality of} $\zeta(s)$\emph{; second, more general, version}$)$.\label{Extendoperatorval}
\hspace*{3mm}Let $K$ be any compact, vertically convex subset of the right critical strip $\{\frac{1}{2}<Re(s)<1\}$, with connected complement in $\mathbb{C}$.\,Assume, for simplicity, that $K$ is symmetric with respect to the real axis.\,Denote by $\mathcal{K}$ the projection of $K$ onto the real axis, and for $c\in \mathcal{K}$, let
\begin{equation}
T(c):=\sup\left(\{T\geq 0:\,[c-iT,c+iT]\subset K\}\right),
\end{equation}
and $T(c)=-\infty$ if there is no such $T$.\,$($By construction, $\mathcal{K}$ is a compact subset of $(\frac{1}{2},1)$ and $0\leq T(c)<\infty$, for every $c\in \mathcal{K}$.$)$\,Further assume that $c\mapsto T(c)$ is continuous on $\mathcal{K}$.\\
\hspace*{3mm}Let $g:K\to \mathbb{C}$ be a non-vanishing $($i.e., nowhere vanishing$)$ continuous function that is holomorphic in the interior of $K$ $($which may be empty$)$.\,Then, given any $\epsilon>0$, there exists $\tau\geq 0$ $($depending only on $\epsilon$$)$ such that

\begin{equation}
\mathcal{J}_{op}(\tau):=\underset{c\in\mathcal{K},\,0\leq T\leq T(c)}{\sup}\Big|\Big|g(\partial_{c}^{(T)})-\zeta(\partial_{c}^{(T)}+i\tau)\Big|\Big|\leq \epsilon,
\end{equation}
where $\partial^{(T)}=\partial_{c}^{(T)}$ is the $T$-truncated infinitesimal shift $($with parameter $c$$)$ and $||.||$ denotes the usual norm in $\mathcal{B}(\mathbb{H}_{c})$ $($the space of bounded linear operators on $\mathbb{H}_{c}$$)$.
\end{theorem}

\hspace*{5mm}\emph{In fact, the set of such $\tau's$ has a positive lower density and, in particular, is infinite.\,More precisely}, \emph{we have} 
\begin{equation}
\underset{\rho\to +\infty}{\liminf}\,\,\frac{1}{\rho}\,vol_{1}(\{\tau\in[0,\rho]:\,\mathcal{J}_{op}(\tau)\leq\epsilon\})>0.
\end{equation}

\begin{proof}
Let $N:=\{c+iT:\,T\in \mathbb{R},\,|T|\leq T(c),\,c\in \mathcal{K}\}$.\,Assume that $T\mapsto T(c)$ is continuous on $\mathcal{K}$.\,Then, we claim that $N$ is a compact subset of $\mathbb{C}$ (and in fact, of $\{s\in \mathbb{C}:\,\frac{1}{2}<Re(s)<1\}$).\\

\hspace*{3mm}In order to justify this claim, we proceed as follows.\,Since $N$ is clearly bounded, then it suffices to show that $N$ is closed.\,Let $(c_{n},T_{n})=c_{n}+iT_{n}$ be an infinite sequence of elements of $N$ such that
\begin{equation}
(c_{n},T_{n})\to(c,T)=c+iT.\notag
\end{equation}
Thus,
\begin{equation}
c_{n}\in \mathcal{K},\,c_{n}\to c \mbox{\quad and\quad} T_{n}\to T.\notag
\end{equation}
As a result, $c\in \mathcal{K}$ (since $\mathcal{K}$ is compact, and hence is closed in $\mathbb{R}$).\,Also, since
\begin{equation}
T_{n}\to T\mbox{\quad as\quad} n\to \infty \mbox{\quad and\quad} |T_{n}|\leq T(c_{n}),\mbox{\quad for all\,} n\geq 1,\notag
\end{equation}
we have 
\begin{equation}
\lim_{n\to \infty}|T_{n}|=|T|\leq \limsup_{n\to \infty}T(c_{n}).\notag
\end{equation}
But since $c_{n}\to c$ and the map $u\mapsto T(u)$ is continuous on $\mathcal{K}$, we have that $T(c_{n})\to T(c)$  as $n\to \infty$.\,Hence, $|T|\leq T(c)$ for any $c\in \mathcal{K}$ and so $(c,T)=c+iT\in N$.\\

\hspace*{3mm}The remainder of the proof of Theorem \ref{Extendoperatorval} proceeds much as in the proof of Theorem \ref{Thm:com}, by applying the extended Voronin theorem (Theorem \ref{Thm:ExtVor}), combined with the functional calculus for bounded normal operators, and using the fact that $\sigma(\partial_{c}^{(T)})=[c-iT, c+iT]$ is contained in $\sigma(\partial_{c}^{(T_{0})})=[c-iT_{0}, c+iT_{0}]$ for any $T$ such that $0<T\leq T_{0}$ and any $c\in\mathcal{K}$.
\end{proof}

\begin{remark}
Instead of assuming that $K$ is symmetric with respect to the real axis, it would suffice to suppose that $c+iT\in K$ $($for some $c\in \mathcal{K}$ and $T>0$$)$ implies that $c-iT \in K$, and vice versa.\\
\end{remark}

\begin{remark}
As in the scalar case $($and taking $K$ to be a line segment$)$, we see that any continuous curve $\big($of $\partial_{c}^{(T)}\big)$ can be approximated by imaginary translates of $\mathfrak{a}^{(T)}=\zeta(\partial_{c}^{(T)})$.\footnote{More precisely, the approximants are not imaginary translates of the truncated spectral operator $\zeta(\partial_{c}^{(T)})$ but instead, they are the results of the Riemann zeta function applied (in the sense of the functional calculus) to imaginary translates of the truncated infinitesimal shifts $\partial_{c}^{(T)}$; namely, $\zeta(\partial_{c}^{(T)}+i\tau)$, for some $\tau\in\mathbb{R}$.}\,Hence, roughly speaking, \emph{the spectral operator} $\mathfrak{a}$ $\big($or its $T$-truncations $\mathfrak{a}^{(T)}\big)$ \emph{can emulate any type of complex behavior:\,it is chaotic}.\footnote{The same cautionary comment as in the previous footnote applies here as well.}\\
\end{remark}

\hspace*{3mm}Note that conditionally $($i.e., under the Riemann hypothesis$)$, and applying the above operator-theoretic version of the universality theorem to $g(s):=\zeta(s)$, we see that, roughly speaking, arbitrarily small scaled copies of the spectral operator are encoded within $\mathfrak{a}$ itself.\,In other words, $\mathfrak{a}$ $($or its $T$-truncation$)$ is \emph{both chaotic and fractal.}

\section{Concluding Comments}

\hspace*{3mm}The universality of the Riemann zeta function $\zeta(s)$ in the right critical strip $\{\frac{1}{2}<Re(s)<1\}$ and its consequences play an important role in other parts of our work in \cite{HerLa1} (see also \cite{HerLa2, HerLa3}).\,In particular, the density of $\zeta(s)$ along the vertical lines $Re(s)=c$ (with $\frac{1}{2}<c<1$), combined with Theorem \ref{Thm:ssop} above (from \cite{HerLa1}), implies that for $\frac{1}{2}<c<1$, the spectrum of the spectral operator $\mathfrak{a}_{c}=\zeta(\partial_{c})$ is equal to the whole complex plane:\,$\sigma(\mathfrak{a}_{c})=\mathbb{C}$.\,By contrast, $\sigma(\mathfrak{a}_{c})$ is a compact subset of $\mathbb{C}$ for $c>1$ and conditionally (i.e., under the Riemann hypothesis), it follows from Theorem \ref{Thm:ssop} (combined with a result of Garunk\v{s}tis and Steuding, see \cite{GarSt}) that $\sigma(\mathfrak{a}_{c})$ is an unbounded, strict subset of $\mathbb{C}$ for $0<c<\frac{1}{2}$.\,The latter result is a consequence of the \textquotedblleft non-universality\textquotedblright of $\zeta(s)$ on the left critical strip $\{0<Re(s)<\frac{1}{2}\}$; see \cite{GarSt} and the relevant references therein.\\

\hspace*{3mm}Our study of the truncated infinitesimal shifts $\partial_{c}^{(T)}$ and their spectra (see Theorem \ref{Thm:Truncv}) has played a crucial role in our proposed quantization of the universality of the Riemann zeta function obtained in Theorems \ref{Thm:com} and \ref{Extendoperatorval}.\,Note that, in light of our functional analytic framework, one should be able to obtain in a similar manner further operator-theoretic versions (or `quantizations') of the known extensions of Voronin's theorem about the universality of $\zeta(s)$ and other $L$-functions $($see, for example, Appendix B, Theorems \ref{Thm:Discruniv}, \ref{Thm:elcurv} and \ref{ThmHurw} below$)$.\,In this broader setting, we expect that the complex variable $s$ should still be replaced by the truncated infinitesimal shifts $\partial^{(T)}=\partial_{c}^{(T)}$, as is the case in \S5 of the present paper.

\section{Appendix A: On the Origins of Universality}

\hspace*{3mm}In 1885, Karl Weierstrass proved that the set of polynomials is dense (for the topology of uniform convergence) in the space of continuous functions on a compact interval of the real line.\,He also proved that the set of trigonometric polynomials is dense (in the above sense) in the class of $2\pi$-periodic continuous functions on $\mathbb{R}$.\,Several improvements of Weierstrass' approximation theorem were obtained by Bernstein (1912), M\"untz (1914), Wiener (1933), Akhiezer--Krein and Paley--Wiener (1934), as well as Stone (1947);\,(see \cite{PerQ} and \cite{St1} for an interesting survey.)\\

\hspace*{3mm}The first `universal' object in mathematical analysis was discovered in 1914 by Feteke.\,He showed the existence of a real-valued power series 
\begin{equation}
\sum_{n=1}^{\infty}a_{n}x^{n}
\end{equation}
which is divergent for all real numbers $x\ne0$.\,Moreover, this divergence is so extreme that, for every continuous function $f$ on $[-1,1]$ such that $f(0)=0$, there exists an increasing sequence $\{N_{k}\}_{K=1}^{\infty}\subset\mathbb{N}$ such that 
\begin{equation}
\lim_{k\to\infty}\sum_{n=1}^{N_{k}}a_{n}x^{n}=f(x),
\end{equation}
uniformly on $[-1,1]$.\\

In 1929, another universal object was discovered  by G.\,D.\,Birkoff.\,He proved that there exists an entire function $f(z)$ such that, for every entire function $g(z)$, there exists a sequence of complex numbers $\{a_{n}\}_{n=1}^{\infty}$ such that
\begin{equation}
\lim_{n\to\infty}f(z+a_{n})=g(z),
\end{equation}
uniformly on all compact subsets of the complex plane.\\

The term of `universality' was used for the first time by J.\,Marcinkiewicz, who obtained the following result:\\

\hspace*{3mm}Let $\{h_{n}\}_{n=1}^{\infty}$ be a sequence of real numbers such that $\lim_{n\to \infty} h_{n}=0$.\,Then, there exists a continuous function $f\in C[0,1]$ such that, for every continuous function $g\in C[0,1]$, there exists an increasing sequence $\{n_{k}\}\subset \mathbb{N}$ such that
\begin{equation}
\lim_{k\to\infty}\frac{f(x+h_{n_{k}})-f(x)}{h_{n_{k}}}=g(x),
\end{equation} 
almost everywhere on $[0,1]$.\,Marcinkiewicz called the function $f$ a \emph{universal primitive}.\\

It is in 1975 that S.\,M.\, Voronin discovered the first explicitly universal object in mathematics, which is the Riemann zeta function $\zeta(s)$.\,His original universality theorem (Theorem \ref{VorThm}) states that any non-vanishing analytic function on the disk (of center $\frac{3}{4}$ and radius less than $\frac{1}{4}$) can be uniformly approximated by imaginary shifts of the Riemann zeta function.\,In the next appendix, we discuss various extensions of this theorem.

\section{Appendix B: On some Extensions of Voronin's Theorem}

\hspace*{3mm}This appendix is dedicated to a discussion of some of the extensions of Voronin's theorem on the universality of the Riemann zeta function.\,(See, e.g., \cite{Bag1, Bag2, Emi, GarSt, Gon, Lau1, Lau2, Lau3, LauMa1, LauMa2, LauMaSt, LauSlez, LauSt, KarVo, Tit, Wil, Rei1, Rei2, St1, St2}.)We begin by mentioning the first improvement of this theorem, which was obtained by Bagchi and Reich in \cite{Bag1, Rei1}, and then discuss further extensions to some other elements of the Selberg class of zeta functions.\,We refer the interested reader to J.\,Steuding's monograph \cite{St1} for a detailed discussion of the historical developments of the notion of universality in mathematics along with various extensions of Voronin's theorem to a large class of $L$-functions.\\

\subsection{A first extension of Voronin's original theorem}
\hspace*{3mm}The first improvement of Voronin's universality theorem was given independently by Bagchi and Reich in \cite{Bag1, Rei1}.\,These authors improved Voronin's theorem by replacing the disk $D$ (see Theorem \ref{VorThm}) by any (suitable) compact subset of the right half of the critical strip (i.e., of $\{\frac{1}{2}<Re(s)<1\}$).\,Their result, often referred to as the extended Voronin (or universality) theorem, has already been stated in \S2.2 above; see Theorem \ref{Thm:ExtVor}.\,By necessity of concision, we will not repeat it here, but will instead make additional (or complementary) comments about some of its consequences and further extensions.\\

\begin{remark}\label{Rk:8.1}
As was already mentioned in Remark \ref{Rk:2.4}, the condition according to which $g(s)$ is non-vanishing is crucial and cannot be dropped.\,Indeed, it can be shown that if a function $g(s)$ $($satisfying the conditions of Theorem \ref{Thm:ExtVor}$)$ were to have at least one zero $($i.e., if $g(s)$ were to vanish somewhere in the compact subset $K$$)$, then a contradiction to the Riemann hypothesis would be obtained; namely, this would imply the existence of a zero of $\zeta(s)$ which is not lying on the critical line $($see, e.g.,  \cite{KarVo, Lau1, St1}$)$.\,Moreover, if we take $g(s)=\zeta(s)$, then the strip of universality is the open right half of the critical strip; namely, $\{\frac{1}{2}<Re(s)<1\}$.\,In other words, it is impossible to extend the universality property of the Riemann zeta function further inside the critical strip.\,Indeed, if such an extension existed on some region $U$, then $U$ would have to intersect the critical line $\{Re(s)=\frac{1}{2}\}$, which $($by Hardy's theorem, see \cite{Tit}$)$ contains infinitely many zeros of $\zeta(s)$.\,This would contradict the assumption according to which the target function $($i.e., $\zeta(s)$$)$ does not have any zeros $($in the given compact set $K$$)$.
\end{remark}

\begin{remark}
In the statement of the extended Voronin theorem $($Theorem \ref{Thm:ExtVor}$)$, the compact set $K$ is allowed to have empty interior.\,In that case, the function $g(s)$ is allowed to be an arbitrary non-vanishing continuous function on $K$.\,Hence, taking $K$ to be a compact subinterval of the real axis $($and taking into account some of the comments in Remark \ref{Rk:8.1}$)$, we conclude that any continuous curve can be approximated by the Riemann zeta function $($and its imaginary translates$)$.\,Further refinements $($and an application of the extended universality theorem to $\zeta(s)$ itself$)$ enable one to see that the graph of $\zeta(s)$ contains arbitrary small \textquotedblleft scaled copies\textquotedblright  of itself, a property characteristic of \textquotedblleft fractality\textquotedblright.\,$($See \cite{Wil}.$)$
\end{remark}

\hspace*{3mm}A variant of Voronin's extended theorem (Theorem \ref{Thm:ExtVor}) about the universality of $\zeta(s)$ was obtained by Reich \cite{Rei1, Rei2}.\,He restricted the approximating shifts of a given target function to arithmetic progressions and obtained a `\emph{discrete}' universality version of Voronin's theorem.\,His result can be stated as follows:

\begin{theorem}\label{Thm:Discruniv}
Let $K$ be a compact subset of the right critical strip $\{\frac{1}{2}<Re(s)<1\}$, with connected complement in $\mathbb{C}$.\,Let $g$ be a non-vanishing continuous function on $K$ which is analytic $($i.e., holomorphic$)$ in the interior of $K$.\,Then, for any $\delta\ne0$ and any $\epsilon>0$,
\begin{equation}
\liminf_{N\to\infty}\frac{1}{N}\#\big(\big\{1\leq n\leq N:\,\max_{s\in K}|g(s)-\zeta(s+in\delta)|<\epsilon\big\}\big)>0,
\end{equation}
where $\#\{.\}$ denotes the cardinality of $\{.\}$.
\end{theorem}

\begin{remark}
To our knowledge, there is no direct relationship between the extended Voronin theorem for the universality of $\zeta(s)$ $($Theorem \ref{Thm:ExtVor}$)$ and Reich's discrete universality theorem $($Theorem \ref{Thm:Discruniv}$)$, in the sense that neither one of them implies the other. 
\end{remark}

\begin{remark}
We refer the reader to $[$\emph{\textbf{St1}}, \S5.7$]$ for a brief exposition of this subject and several additional references.
\end{remark}

\subsection{Further extensions to L-functions}

\hspace*{3mm}In this appendix, we briefly discuss some of the extensions of the universality of the Riemann zeta function to other elements of the Selberg class of zeta functions as well as to other types of zeta functions not belonging to the Selberg class.\footnote{A survey of some of the key definitions and properties of the Selberg class can be found in [\textbf{La5}, Appendix E], where many relevant references are also provided.}\\

\hspace*{3mm}Eminyan obtained an extension of Voronin's original theorem (Theorem \ref{VorThm}) to a class of Dirichlet-type $L$-functions whose Euler product are defined in terms of a finite number of primes; see \cite{Emi}.\\

\hspace*{3mm}Later on, an extension to the class of $L$-functions associated to cusp forms was obtained by Laurincikas and Matsumoto in \cite{LauMa2}.\,They obtained a universality theorem for $L$-functions attached to normalized eigenforms of the full modular group.\\

\hspace*{3mm}A further extension to new forms (or elliptic curves associated to modular forms) was obtained by Laurincikas, Matsumoto and Steuding in \cite{LauMaSt}.\,We next briefly state their result:\footnote{See, e.g., [\textbf{La5}, Appendix C] and the many relevant references therein for the terminology and definitions about modular forms which are used here.}

\begin{theorem}\label{Thm:elcurv}
Suppose that $f$ is a new form of weight $k$ and level $N$.\,Let $K$ be a compact subset of the strip $\{\frac{k}{2}<Re(s)<\frac{k+1}{2}\}$ with connected complement, and let $g(s)$ be a continuous non-vanishing function on $K$ which is analytic $($i.e., holomorphic$)$ in the interior of $K$.\,Then, given any $\epsilon>0$, there exists $\tau>0$ such that
\begin{equation}
\sup_{s\in K}|g(s)-L(s+i\tau,f)|<\epsilon.
\end{equation}

\hspace*{3mm} Moreover, the set of admissible $\tau$'s is infinite and, in fact, 
\begin{equation}
\liminf_{T\to \infty}vol_{1}\big(\{\tau\in[0,T]:\,\sup_{s\in K}|L(s+i\tau,f)-g(s)|<\epsilon\}\big)>0.
\end{equation}
\end{theorem}

\hspace*{3mm}An extension to Dirichlet series with multiplicative coefficients was obtained by Laurincikas and Slezeviciene in \cite{LauSlez}.\\

\hspace*{3mm}A number of additional results concerning the extensions and applications of the universality theorem to $L$-functions can be found in Steuding's monograph \cite{St1}.\\

\hspace*{3mm}Finally, we point out the fact that extensions to other families of zeta functions not necessarily belonging to the Selberg class of zeta functions (such as the family of Hurwitz zeta functions) were also obtained.\,Given $\alpha\in(0,1]$ and for $Re(s)>1$, the Hurwitz zeta function is defined by
\begin{equation}\label{Eq:Hurwz}
\zeta(s,\alpha)=\sum_{m=0}^{\infty}\frac{1}{(m+\alpha)^{s}}.
\end{equation} 

\hspace*{3mm}This function has a meromorphic continuation to the whole complex plane.\,It has a simple pole at $s=1$ with residue equal to 1.\,Note that for $\alpha=1$, we have $\zeta(s,1)=\zeta(s)$, the Riemann zeta function, and that for $\alpha=\frac{1}{2}$, $\zeta(s,\frac{1}{2})$ is given (for all $s\in\mathbb{C}$) by
\begin{equation}
\zeta(s,\frac{1}{2})=(2^{s}-1)\zeta(s).
\end{equation}
Clearly, for $\alpha=1$ and $\alpha=\frac{1}{2}$, $\zeta(s,\alpha)$ has an Euler product expansion.\,This is \emph{not} the case, however, for $\alpha\in(0,1]-\{\frac{1}{2},\,1\}$.\,As a result, except for $\alpha\in\{\frac{1}{2},\,1\}$, the Hurwitz zeta function, defined by Equation (\ref{Eq:Hurwz}), is not an element of the Selberg class of zeta functions.\\

\hspace*{3mm}An extension of Theorem \ref{Thm:ExtVor} to the class of Hurwitz zeta functions for the case $\alpha\in(0,1]-\{\frac{1}{2},1\}$, where $\alpha$ is \emph{rational} or \emph{transcendental}, was obtained independently by Gonek in \cite{Gon} and Bagchi in \cite{Bag1}.\,Their result can be stated as follows:

\begin{theorem}\label{ThmHurw}
Suppose $\alpha\in(0,1]-\{\frac{1}{2},1\}$ is either \emph{rational} or \emph{transcendental}.\,Let $K\subset\{\frac{1}{2}<Re(s)<1\}$ be a compact subset with connected complement.\,Let $g(s)$ be a continuous function on $K$ which is analytic $($i.e., holomorphic$)$ in the interior of $K$.\,Then, for every $\epsilon>0$, we have
\begin{equation}\label{Hurwapp}
\liminf_{T\to+\infty}vol_{1}\big(\{\tau\in[0,T]:\,\sup_{s\in K}|g(s)-\zeta(s+i\tau,\alpha)|<\epsilon\}\big)>0.
\end{equation}
\end{theorem} 

\begin{remark}
Note that in the statement of Theorem \ref{ThmHurw} and in contrast to the case of the Riemann zeta function $($and other $L$-functions$)$, the function $g$ is allowed to have zeros inside the compact set $K$.\footnote{This is not all that surprising since except for $\alpha\in\{\frac{1}{2},1\}$, the Hurwitz zeta function is not expected to satisfy the Riemann hypothesis.}
\end{remark}

\hspace*{3mm}In view of Equation $($\ref{Hurwapp}$)$, the Hurwitz zeta function $\zeta(s, \alpha)$ can uniformly approximate target functions which may have zeros inside the compact subset $K$.\,Hence, the Hurwitz zeta function is an example of a mathematical object which is `\emph{strongly universal}'.\,The theory of strong universality has been developed in several directions.\,We note that the results obtained within our functional analytic framework about the truncated infinitesimal shifts $\partial_{c}^{(T)}$ and the truncated spectral operators $\mathfrak{a}_{c}^{(T)}=\zeta(\partial_{c}^{(T)})$ can also be used to provide an operator-theoretic extension of the notion of \emph{strong universality}.

\section{Appendix C: Almost Periodicity and the Riemann Hypothesis}

\hspace*{3mm}Let $f$ be a holomorphic complex-valued function on some vertical strip $S_{a,\,b}=\{s\in\mathbb{C}:\,a<Re(s)<b\}$.\,Then, $f$ is said to be \emph{almost periodic} if for every $\epsilon>0$ and any $\alpha$, $\beta$ such that $a<\alpha<\beta<b$, there exists $\ell(f, \alpha, \beta, \epsilon)>0$ such that in every interval $(t_{1}, t_{2})$ of length $\ell$, there exists a number $\tau\in(t_{1}, t_{2})$ such that for any $\alpha\leq x\leq \beta$ and any $y\in\mathbb{R}$, we have 
\begin{equation}
|f(x+iy+i\tau)-f(x+iy)|<\epsilon.\notag
\end{equation}

\hspace*{3mm}The notion of almost periodicity was introduced by H.\,Bohr in \cite{Boh1}.\,He proved that any Dirichlet series is almost periodic in its half-plane of absolute convergence.\,Moreover, he showed in \cite{Boh2} that the almost periodicity of the class of Dirichlet $L$-functions $L(s,\chi)$ with \emph{non-trivial} character $\chi$ (i.e., $\chi\ne1$) is intimately connected with the location of the critical zeros of the Riemann zeta function:\footnote{Recall that the Dirichlet $L$-function (or Dirichlet $L$-series) is initially defined by $L(s,\chi):=\sum_{n=1}^{\infty}\frac{\chi(n)}{n^{s}}$ for $Re(s)>1$; see, e.g., \cite{Pat}, \cite{Ser}.}

\begin{theorem}\label{Thm:9.1}
Given any character $\chi\ne1$, then $L(s, \chi)$ is almost periodic in the half-plane $\{Re(s)>\frac{1}{2}\}$ if and only if the Riemann hypothesis is true.
\end{theorem}

\begin{remark}
The notion of almost periodicity was introduced by H.\,Bohr as an analytic tool for proving the Riemann hypothesis.\,We note that his approach failed for the case of the Riemann zeta function but led to a reformulation of the Riemann hypothesis for the class of Dirichlet $L$-functions associated to a non-trivial character.\,$($See Theorem \ref{Thm:9.1} just above.$)$\,In contrast to the Riemann zeta function, which has a pole at $s=1$ and whose Dirichlet series and Euler product converge only for $Re(s)>1$, the Dirichlet $L$-functions with non-trivial characters are holomorphic for $Re(s)>0$ $($even in all of $\mathbb{C}$ if the Dirichlet character $\chi$ is primitive$)$ and have a Dirichlet series and an Euler product which converge for $Re(s)>0$ and in particular, inside the critical strip\emph{:}
\begin{equation}
L(s, \chi)=\sum_{n=1}^{\infty}\frac{\chi(n)}{n^{s}}=\prod_{p\in\mathcal{P}}(1-\chi(p)p^{-s})^{-1},\mbox{\quad for\,\,}Re(s)>0.\notag
\end{equation} 

\hspace*{3mm}A key fact, also established by Bohr $($and central to his reformulation of RH$)$ is that $($still for a nontrivial character$)$ $L(s,\,\chi)$ is almost periodic in the critical strip $($and actually, for $Re(s)>0$$)$ because its Euler product is convergent there.\footnote{The Euler product converges absolutely for $Re(s)>1$ and conditionally for $Re(s)>0$; see, e.g., \cite{Ser}.\,It is the convergence $($and not the absolute convergence$)$ which is essential here.}
\end{remark}

\begin{remark}
The concept of almost periodicity is key to the proof and understanding of the universality of $\zeta$ and of other $L$-functions.\,In fact, toward the beginning of the 20th century, Bohr's theory of almost periodicity was already used by Harald Bohr and his collaborators in order to obtain several interesting results concerning the Riemann zeta function $($and other Dirichlet $L$-functions$)$, such as the density of the range of $\zeta(s)$ along the vertical lines $\{Re(s)=c\}$, with $\frac{1}{2}<c<1$ $($see Theorem \ref{Thmdens} and the discussion preceding it in \S2.3$)$.\,These results and their extensions, which eventually led to Voronin's universality theorem $($and its various generalizations$)$, make use of finite Euler products $($for $\zeta(s)$, say$)$ within the critical strip $\{0<Re(s)<1\}$.\,$($See, e.g., \cite{St1} for a detailed exposition.$)$
\end{remark}

\bibliographystyle{amsalpha}

\end{document}